\documentclass{article}

\usepackage[T1]{fontenc}
\usepackage[utf8]{inputenc}
\usepackage{microtype}
\usepackage[letterpaper,left=1in,right=1in,top=0.8in,bottom=1in]{geometry}
\usepackage[bottom]{footmisc}
\usepackage{mathtools}
\usepackage{amssymb}
\usepackage{amsthm}
\usepackage{thmtools}
\usepackage{mdframed}
\allowdisplaybreaks[4]

\usepackage[table,dvipsnames]{xcolor}
\definecolor{bgcolor}{rgb}{0.8,1,1}
\definecolor{bgcolor2}{rgb}{0.8,1,0.8}
\definecolor{niceblue}{rgb}{0.0,0.19,0.56}
\definecolor{shadecolor}{gray}{0.9}
\usepackage{booktabs}
\usepackage{makecell}
\usepackage{graphicx}
\usepackage{subcaption}
\usepackage[shortlabels]{enumitem}
\usepackage{pifont}

\usepackage{algorithm}
\usepackage[noEnd=false]{algpseudocodex}
\usepackage[normalem]{ulem}
\usepackage{soul}

\usepackage[sort]{natbib}
\usepackage[hypertexnames=false,colorlinks,linkcolor={blue!75!black},%
            urlcolor=blue,citecolor={blue!75!black}]{hyperref}
\usepackage{cleveref}

\declaretheoremstyle[
  headfont=\normalfont\bfseries,
  notefont=\mdseries, notebraces={(}{)},
  bodyfont=\normalfont,
  postheadspace=0.5em,
  spaceabove=1pt,
  mdframed={
    skipabove=8pt,
    skipbelow=8pt,
    hidealllines=true,
    backgroundcolor={shadecolor},
    innerleftmargin=4pt,
    innerrightmargin=4pt}
]{shaded}

\declaretheorem[style=shaded,within=section]{definition}
\declaretheorem[style=shaded,sibling=definition]{theorem}
\declaretheorem[style=shaded,sibling=definition]{proposition}

\declaretheorem[style=shaded,sibling=definition]{lemma}
\declaretheorem[style=shaded,sibling=definition]{remark}

\newtheorem*{problem*}{Problem}

\crefname{algorithm}{algorithm}{algorithms}
\Crefname{algorithm}{Algorithm}{Algorithms}
\newcommand{\R}{\mathbb{R}}
\newcommand{\norm}[1]{\left\lVert #1\right\rVert}
\newcommand{\ip}[2]{\left\langle #1,#2\right\rangle}

\DeclareMathOperator{\sign}{sign}

\newcommand{\cmark}{\ding{51}}
\newcommand{\xmark}{\ding{55}}
\sethlcolor{yellow}

\title{\textbf{Curvature-Scaled Sharpness-Aware Minimization:\\ Convergence Beyond Smoothness}}

\author{
    Dimitris Oikonomou$^{1,2}$\\
    \texttt{doikono1@jh.edu} 
    \and 
    Nicolas Loizou$^{1,3}$\\
    \texttt{nloizou@jhu.edu}
}

\date{
    \vspace{1em}
    $^1$Mathematical Institute for Data Science (MINDS), Johns Hopkins University\\
    $^2$Department of Computer Science, Johns Hopkins University\\
    $^3$Department of Applied Mathematics and Statistics, Johns Hopkins University\\
    \vspace{0.5em}
}

\begin{document}

\maketitle

\begin{abstract}
    \noindent
    Sharpness-Aware Minimization (SAM) has established itself as a widely used method for training deep neural networks. Despite its popularity, existing convergence guarantees for SAM-type methods require smoothness, a restrictive assumption. Generalized smoothness conditions have recently been proposed in the analysis of gradient methods but remain underexplored for SAM. To analyze the effect of SAM's perturbation under generalized smoothness, we use relative gradient distortion to quantify how the sharpness-aware direction departs from the original gradient. Under $(L_0,L_1)$-smoothness, we show that this distortion can become unbounded near stationarity for SAM and at large gradients for its unnormalized variant (USAM), with examples where SAM cycles and USAM diverges. We then establish a necessary condition for uniform distortion control among perturbations along the gradient whose nonnegative multipliers depend only on its norm. Motivated by this condition, we propose Curvature-Scaled SAM (CS-SAM), which scales both the perturbation and the descent step by suitable curvature upper bounds, and prove that it uniformly bounds the distortion. Under both $(L_0,L_1)$- and $(H_0,H_1)$-smoothness, we establish deterministic $\mathcal{O}(1/T)$ rates for the best squared gradient norm on nonconvex objectives, last-iterate $\mathcal{O}(1/T)$ rates for the function gap on convex objectives with a minimizer, and linear convergence under the Polyak--\L{}ojasiewicz condition. With zero perturbation, our guarantees recover the rates for gradient descent with $(L_0,L_1)$- or $(H_0,H_1)$-based step-sizes. Minibatch experiments on computer vision tasks demonstrate competitive performance for both variants, with the $(H_0,H_1)$ variant achieving the highest mean best-test accuracy in all four settings against tuned SAM and USAM with constant and cosine annealing schedules.
\end{abstract}

\section{Introduction}
\label{sec:intro}

Sharpness-Aware Minimization (SAM) was introduced to improve generalization by seeking parameters whose training loss remains low throughout a local neighborhood \citep{foret2021sam}. It pursues this objective through a simple two-step procedure: first perturb the parameters in a direction of increasing loss, then update the original parameters using the gradient evaluated at the perturbed point. The choice of perturbation therefore determines how the resulting update departs from ordinary gradient descent. For a differentiable objective $f:\R^d\to\R$, standard SAM uses a fixed-radius perturbation of length $r$ along the normalized gradient, whereas its unnormalized variant (USAM) uses a fixed multiple $\rho$ of the gradient itself \citep{andriushchenko2022understanding}:
\begin{equation}
    y_t=x_t+r\frac{\nabla f(x_t)}{\norm{\nabla f(x_t)}},
    \qquad x_{t+1}=x_t-\eta_t\nabla f(y_t),
    \tag{SAM}\label{eq:sam}
\end{equation}
\begin{equation}
    y_t=x_t+\rho\nabla f(x_t),
    \qquad x_{t+1}=x_t-\eta_t\nabla f(y_t),
    \tag{USAM}\label{eq:usam}
\end{equation}
with radius $r>0$, multiplier $\rho\ge0$, and outer step $\eta_t>0$. Both rules cost two gradient evaluations per iteration, one at the iterate $x_t$ and one at the perturbed point $y_t$. 

\medskip
\noindent
Throughout, we consider the problem
\begin{equation}
    \min_{x\in\R^d} f(x),
    \qquad f_\star:=\inf_{x\in\R^d}f(x)>-\infty,
    \label{eq:problem}
\end{equation}
where $f:\R^d\to\R$ is twice continuously differentiable. In model training, $x$ represents the model parameters and $f$ the training loss. We call the Hessian operator norm $\norm{\nabla^2 f(x)}$ the local curvature at $x$. Classical smoothness bounds this quantity by a constant $L>0$:
\begin{equation}
    \norm{\nabla^2 f(x)}\le L
    \qquad\forall x\in\R^d.
    \label{eq:assSmooth}
\end{equation}
Classical $L$-smoothness can fail even for simple neural-network training objectives \citep{alimisis2026warmup}. Motivated by empirical observations that local curvature grows with the gradient norm during neural-network training, \citet{zhang2020clipping} introduced $(L_0,L_1)$-smoothness, which allows the Hessian norm to grow with the gradient norm. Subsequent theoretical results support this condition for specific neural-network models. In particular, the bounds of \citet{taheri2023fast} imply $(L_0,L_1)$-smoothness for two-layer networks with fixed output weights and smoothed leaky-ReLU activations, trained with exponential loss on linearly separable data. Complementary results establish $(H_0,H_1)$-smoothness for two-layer networks trained with $\ell_2$-regularized cross-entropy loss under suitable activation assumptions \citep{alimisis2026warmup}. These findings motivate the gradient-dependent and loss-dependent curvature bounds considered here:
\begin{align}
    \norm{\nabla^2 f(x)}&\le M_L(x):=L_0+L_1\norm{\nabla f(x)},
    &\text{$(L_0,L_1)$-smoothness},\label{eq:assL}\\
    \norm{\nabla^2 f(x)}&\le M_H(x):=H_0+H_1\bigl(f(x)-f_\star\bigr),
    &\text{$(H_0,H_1)$-smoothness},\label{eq:assH}
\end{align}
for every $x\in\R^d$, with $L_0>0$, $L_1\ge0$, and $H_0,H_1\ge0$. Note that both models contain classical smoothness as a special case ($L_1=0$ or $H_1=0$). With slight abuse of notation, in several parts of the paper we write $M_L:=M_L(x)$ and $M_H:=M_H(x)$. We refer to a positive function that upper-bounds the local curvature everywhere as a \emph{curvature scale}; $M_L$ and, for nonconstant objectives, $M_H$ provide the two scales studied in this work (\Cref{prop:admissible}).

\medskip
\noindent
The difficulty is that the curvature bound must control the gradient change along the perturbation, as well as the subsequent descent step. Even under classical smoothness, a fixed perturbation radius can prevent exact convergence: \citet{bartlett2022dynamics} analyzed cycling of fixed-radius \ref{eq:sam} on convex quadratics, and \citet{si2023practical} established stationarity neighborhoods in several settings. The inexact-gradient viewpoint of \citet{khanh2024fundamental} makes the perturbation-induced error explicit, while the clipped rule of \citet{adly2026stationarity} controls it near stationarity by making the perturbation proportional to the gradient. Under \eqref{eq:assL}, however, a gradient-proportional perturbation can also become too large when the gradient grows. Thus, extending SAM to generalized smoothness requires control at both ends of the gradient scale.

\medskip
\noindent
For ordinary gradient methods, adapting to curvature growth is already well studied. The $(L_0,L_1)$ model was introduced by \citet{zhang2020clipping} to explain gradient clipping and has since supported analyses of clipped, normalized, and adaptive methods \citep{zhang2020improved,crawshaw2022robustness,chen2023generalized,koloskova2023revisiting,li2023convex}. In particular, \citet{gorbunov2025methods} and \citet{vankov2025optimizing} analyzed steps inversely proportional to gradient-dependent curvature bounds. The loss-dependent model has likewise been used to explain learning-rate warm-up \citep{alimisis2026warmup} and to analyze line search and adaptive gradient methods \citep{vaswani2025armijo,vaswani2026nonuniform}. These results motivate scaling the outer descent step by a curvature bound. The remaining question for SAM is how that bound should also determine the perturbation.

\subsection{Main Contributions}

\paragraph{$\diamond$ Complementary failures and a necessary condition for distortion control.}
In \Cref{sec:distortion}, we show that the relative gradient distortion of \ref{eq:sam} can become unbounded as the gradient vanishes, while that of \ref{eq:usam} can become unbounded as the gradient grows. \Cref{prop:endpoint-failures} gives explicit examples within the $(L_0,L_1)$-smooth class: fixed-radius \ref{eq:sam} admits a nonzero two-cycle on a quadratic, and \ref{eq:usam} diverges from sufficiently large initializations on a quartic. Motivated by these complementary failures, \Cref{thm:necessary-scaling} establishes a necessary condition for uniformly bounded distortion over this class among perturbations along the gradient whose nonnegative multipliers depend only on its norm. The perturbation length must be at most proportional to the gradient norm near stationarity and bounded at large gradients. This condition motivates the gradient-dependent perturbation of CS-SAM.

\paragraph{$\diamond$ Curvature-Scaled Sharpness-Aware Minimization.}
In \Cref{sec:cssam}, we construct \ref{eq:cssam}, which, for a curvature scale $M$, iterates
\begin{equation}
    y_t=x_t+\frac{\rho}{M(x_t)}\,\nabla f(x_t),
    \qquad
    x_{t+1}=x_t-\frac{\gamma}{M(y_t)}\,\nabla f(y_t),
    \tag{CS-SAM}
\end{equation}
with perturbation multiplier $\rho/M(x_t)$ and descent step $\gamma/M(y_t)$, where $\rho\ge0$ and $\gamma>0$. We identify conditions on the curvature scale $M$ that ensure uniform distortion control and show that both $M_L$ and $M_H$ satisfy them. For every scale satisfying these conditions, \Cref{thm:relative-control} bounds the distortion by $e^\rho-1$ at every nonstationary iterate. 

\paragraph{$\diamond$ Convergence under generalized smoothness.}
We establish convergence guarantees for \ref{eq:cssam} under $(L_0,L_1)$-smoothness with $M=M_L$ and under $(H_0,H_1)$-smoothness with $M=M_H$ (\Cref{thm:L-rates,thm:H-rates}). In each setting, for suitable perturbation and step parameters, we prove $\mathcal{O}(1/T)$ rates for the best squared gradient norm on nonconvex objectives and last-iterate $\mathcal{O}(1/T)$ rates for the function gap on convex objectives with a minimizer. We also prove linear convergence in both settings under the Polyak--\L{}ojasiewicz (PL) condition, $\norm{\nabla f(x)}^2/2\ge\mu(f(x)-f_\star)$ for every $x$ and some $\mu>0$.

\paragraph{$\diamond$ Recovery of gradient descent guarantees for $\rho=0$.}
Setting $\rho=0$ recovers gradient descent with step $\gamma/M(x_t)$. Our bounds recover the known nonconvex and convex rates for gradient-dependent steps \citep{vankov2025optimizing,gorbunov2025methods}, and the convex and PL rates for loss-dependent steps \citep{alimisis2026warmup}, while allowing larger step sizes in the latter case. Thus the analysis extends existing gradient descent guarantees to positive sharpness-aware perturbations.

\paragraph{$\diamond$ Numerical Evaluation.}
In \Cref{sec:verification}, we verify \Cref{prop:endpoint-failures} and \Cref{thm:relative-control}: \ref{eq:sam} cycles and \ref{eq:usam} diverges, whereas \ref{eq:cssam} converges with a distortion below $e^\rho-1$ for both scales. In \Cref{sec:dnn}, we evaluate minibatch variants beyond the deterministic setting of our theory on ResNet-20/32 and CIFAR-10/100. Both use two gradient evaluations per iteration. Against tuned constant-step and cosine annealing-scheduled \ref{eq:sam} and \ref{eq:usam}, the loss-dependent variant achieves the highest mean best-test accuracy in all four settings, while the gradient-dependent variant exceeds all these baselines in three of four settings.

\medskip
\noindent
\Cref{tab:comparison} summarizes the distortion control and convergence guarantees alongside prior work. 

\begin{table}[t]
    \centering
    \caption{Comparison with the closest prior work. $D(x,d)=\norm{\nabla f(x+d)-\nabla f(x)}/\norm{\nabla f(x)}$
    is the relative gradient distortion of the perturbation $d$ (\Cref{sec:distortion}). Curvature
    models: $L$ means $\norm{\nabla^2 f(x)}\le L$; $(L_0,L_1)$ means
    $\norm{\nabla^2 f(x)}\le L_0+L_1\norm{\nabla f(x)}$; $(H_0,H_1)$ means
    $\norm{\nabla^2 f(x)}\le H_0+H_1(f(x)-f_\star)$, with $f_\star=\inf_x f(x)$. \cmark: $D$ is bounded by
    a constant; \xmark: the bound degenerates. Rates are on the best squared gradient norm over $T$ iterations (nonconvex), on
    the last-iterate gap $f(x_T)-f_\star$ for convex objectives with a minimizer (convex), and on the function gap
    under PL.}
    \label{tab:comparison}
    \resizebox{\linewidth}{!}{%
    \begin{tabular}{l l c c c c c}
    \toprule
    Method & Perturbation $d(x)$ & \makecell{Curvature\\model} & \makecell{$D$ uniformly\\bounded} & Nonconvex & Convex & PL \\
    \midrule
    \makecell[l]{\ref{eq:sam} \citep{foret2021sam}\\rates: \citep{si2023practical},\\ \citep{oikonomou2025general}} & $r\,\nabla f(x)/\norm{\nabla f(x)}$ & $L$ & \xmark & $\mathcal{O}(1/T)+L^2r^2$ & \makecell{no last-iterate\\gap rate} & \makecell{linear to a\\neighborhood} \\
    \addlinespace
    
    \makecell[l]{\ref{eq:usam}\\\citep{andriushchenko2022understanding}} & $\rho\,\nabla f(x)$ & $L$ & \cmark & $\mathcal{O}(1/T)$ & $\mathcal{O}(1/T)$ & linear \\
    \addlinespace
    
    \makecell[l]{Unified SAM\\\citep{oikonomou2025general}} & $\rho[(1-\lambda)+\lambda/\norm{\nabla f(x)}]\cdot\nabla f(x)$ & $L$ & \makecell{\cmark\ if $\lambda=0$;\\\xmark\ otherwise} & \makecell{$\mathcal{O}(T^{-1/2})$\\(in expectation)} & not analyzed & \makecell{$\lambda=0$: linear\\$\lambda>0$: linear to a \\neighborhood} \\
    \addlinespace
    
    \makecell[l]{Clipped perturbation\\\citep{adly2026stationarity}} & \makecell[l]{$\rho_k\nabla f(x)/\norm{\nabla f(x)}^{\alpha}$,\\$\rho_k=\min\{\rho,\frac{\theta}{L}\norm{\nabla f(x)}^\alpha\}$} & $L$ & \cmark & $\mathcal{O}(1/T)$ & not analyzed & not analyzed \\
    \addlinespace
    
    \makecell[l]{Inexact-gradient SAM\\\citep{khanh2024fundamental}} & \makecell[l]{normalized and\\unnormalized variants} & $L$ & \makecell{normalized: \xmark\\USAM: \cmark} & \makecell{$\norm{\nabla f(x_t)}\to0$\\(no generic rate)} & \makecell{$f(x_T)\to f_\star$\\(no rate)} & \makecell{USAM: linear under\\KL exponent $1/2$} \\
    \midrule
    \rowcolor{bgcolor2} Ours, \ref{eq:cssam} with $M_L$ & $\rho\,\nabla f(x)/(L_0+L_1\norm{\nabla f(x)})$ & $(L_0,L_1)$ & \cmark & $\mathcal{O}(1/T)$ & $\mathcal{O}(1/T)$ & linear \\
    \rowcolor{bgcolor2} Ours, \ref{eq:cssam} with $M_H$ & $\rho\,\nabla f(x)/(H_0+H_1(f(x)-f_\star))$ & $(H_0,H_1)$ & \cmark & $\mathcal{O}(1/T)$ & $\mathcal{O}(1/T)$ & linear \\
    \bottomrule
    \end{tabular}
    }
\end{table}

\section{From gradient distortion to curvature-scaled SAM}
\label{sec:failures}

In this section, we define relative gradient distortion and explain how it controls the difference between a sharpness-aware direction and the gradient used by gradient descent. We then show that \ref{eq:sam} and \ref{eq:usam} can lose this control at opposite ends of the gradient scale, and derive a necessary condition for keeping it uniformly bounded. This condition motivates the CS-SAM update, for which we prove a uniform distortion bound under both generalized smoothness models. The proofs are deferred to \Cref{app:endpoint}; the distortion bound follows from the auxiliary lemmas in \Cref{app:lemmas}.

\subsection{Relative gradient distortion and the failures of SAM and USAM}
\label{sec:distortion}

The outer step of \ref{eq:sam} and \ref{eq:usam} is a gradient step that uses $\nabla f(y_t)$ in place of $\nabla f(x_t)$, so a descent analysis must account for the effect of this substitution as well as the step size. We now measure the relative error of this substitution. For $\nabla f(x)\ne0$ and a displacement $d$, we define the \emph{relative gradient distortion}
\begin{equation}
    D(x,d):=\frac{\norm{\nabla f(x+d)-\nabla f(x)}}{\norm{\nabla f(x)}}.
    \label{eq:distortion}
\end{equation}
Let us explain why this is the right quantity. If $D(x,d)\le\delta<1$, then $\nabla f(x+d)$ makes an acute angle with $\nabla f(x)$, since $\ip{\nabla f(x)}{\nabla f(x+d)}\ge(1-\delta)\norm{\nabla f(x)}^2$, and has comparable length, since $(1-\delta)\norm{\nabla f(x)}\le\norm{\nabla f(x+d)}\le(1+\delta)\norm{\nabla f(x)}$. These are exactly the two properties the descent lemma requires, so $D(x,d)\le\delta<1$ is the property we want a perturbation to have, and we want it at \emph{every} point. Controlling the perturbation-induced gradient error is also the viewpoint of the inexact-gradient analyses of \citet{khanh2024fundamental}, and relative control of this error near stationarity is what the clipped rule of \citet{adly2026stationarity} enforces; the quantity itself is therefore not new. Our question is which perturbation rules keep it uniformly bounded under generalized smoothness. 

\medskip
\noindent
We first record what the two base rules can guarantee. Under \eqref{eq:assSmooth} we have
\begin{equation}
    D_{\mathrm{SAM}}\le \frac{Lr}{\norm{\nabla f(x)}},
    \qquad
    D_{\mathrm{USAM}}\le L\rho.
    \label{eq:heuristic-bounds-smooth}
\end{equation}
Note that the bound for SAM degenerates as $\norm{\nabla f(x)}\to0$, while the bound for USAM is uniform. Under the more general \eqref{eq:assL}, for $L_1>0$ (both bounds are proved in \Cref{app:endpoint}):
\begin{equation}
    D_{\mathrm{SAM}}\le \bigl(e^{L_1r}-1\bigr)\left(1+\frac{L_0}{L_1\norm{\nabla f(x)}}\right),
    \qquad
    D_{\mathrm{USAM}}\le \frac{L_0+L_1\norm{\nabla f(x)}}{L_1\norm{\nabla f(x)}}
    \bigl(e^{L_1\rho\norm{\nabla f(x)}}-1\bigr).
    \label{eq:heuristic-bounds-l0l1}
\end{equation}
The bound for SAM tends to the constant $e^{L_1r}-1$ at large gradients and degenerates as $\norm{\nabla f(x)}\to0$. The bound for USAM tends to $\rho L_0$ as $\norm{\nabla f(x)}\to0$ and grows exponentially as $\norm{\nabla f(x)}\to+\infty$. This suggests that SAM's distortion can explode when the gradient vanishes, while USAM's distortion can explode when the gradient grows. The following proposition makes the two failures precise.

\begin{proposition}
    \label{prop:endpoint-failures}
    Both perturbations can fail on one-dimensional convex objectives: \ref{eq:sam} admits a nonzero two-cycle on a quadratic, and \ref{eq:usam} diverges for large initializations on a quartic.
    \begin{enumerate}[label=(\roman*), align=left, leftmargin=*, widest=ii]
        \item
        Let $f(x)=x^2/2$, which satisfies \eqref{eq:assSmooth} with $L=1$. Run \ref{eq:sam} with $r>0$ and $\eta\in(0,2)$. For every $x\ne0$,
        \begin{equation}
            D_{\mathrm{SAM}}(x):=D(x,r\sign(x))=\frac{r}{|x|},
            \qquad
            x_{t+1}=\bigl[1-\eta\,(1+D_{\mathrm{SAM}}(x_t))\bigr]\,x_t .
        \end{equation}
        Let $a=\eta r/(2-\eta)$. Then $|x_{t+1}|<|x_t|$ if $|x_t|>a$, $|x_{t+1}|>|x_t|$ if $0<|x_t|<a$, and $x_{t+1}=-x_t$ if $|x_t|=a$. Consequently $\pm a$ is a two-cycle, and no trajectory converges to the minimizer (unless it lands on it exactly).

        \item
        Let $f(x)=x^4/4$, which satisfies \eqref{eq:assL} with $(L_0,L_1)=(4,1)$. Run \ref{eq:usam} with $\rho>0$ and $\eta>0$. For every $x\ne0$,
        \begin{equation}
            D_{\mathrm{USAM}}(x):=D(x,\rho x^3)=(1+\rho x^2)^3-1,
            \qquad
            x_{t+1}=\bigl[1-\eta\,(1+D_{\mathrm{USAM}}(x_t))\,x_t^2\bigr]\,x_t .
        \end{equation}
        Consequently, if $|x_0|>(2/(\eta\rho^3))^{1/8}$, then $x_{t+1}$ has the opposite sign of $x_t$ and $|x_{t+1}|>|x_t|$ for every $t$, and $|x_t|\to\infty$.
    \end{enumerate}
\end{proposition}
The quadratic construction can be viewed as an exact illustration of established late-stage phenomena: \citet{bartlett2022dynamics} showed that fixed-radius SAM typically approaches a two-cycle on convex quadratics, and \citet{si2023practical} established unavoidable stationarity neighborhoods for practical fixed-radius SAM in several settings. Our point is the complementarity: the distortion can become unbounded at opposite ends of the gradient scale, with nonconvergence for these fixed-parameter updates. The following theorem addresses all perturbations along the gradient whose nonnegative multiplier depends only on the gradient norm: it gives a necessary condition for uniform boundedness of the relative distortion over the whole class of objectives in \eqref{eq:assL}.

\begin{theorem}
    \label{thm:necessary-scaling}
    Let $L_0>0$, $L_1>0$, and consider the perturbation given by a nonnegative collinear rule $d(x)=\alpha(\norm{\nabla f(x)})\nabla f(x)$. Suppose that for every one-dimensional objective satisfying \eqref{eq:assL} and every point with nonzero gradient, $D(x,d(x))\le\delta$. Then, for every $s>0$,
    \begin{equation}
        \alpha(s)\le \min\left\{\frac{\delta}{L_0},
        \frac{\log(1+\delta)}{L_1s}\right\}.
        \label{eq:necessary}
    \end{equation}
\end{theorem}
Note that the theorem covers \ref{eq:sam} with $\alpha(s)=r/s$ and \ref{eq:usam} with $\alpha(s)=\rho$. The bound says that the perturbation length $\alpha(s)s$ must be at most linear in the gradient near $s=0$ and at most a constant for large $s$. \ref{eq:sam} violates the first requirement since $r/s\to\infty$ as $s\to0$, and \ref{eq:usam} violates the second since $\rho s\to\infty$ as $s\to\infty$. In words, every perturbation of this form with uniformly bounded distortion must be USAM-like near stationarity and SAM-like at large gradients. Dividing $\rho\nabla f(x)$ by $L_0+L_1\norm{\nabla f(x)}$ gives the multiplier $\alpha(s)=\rho/(L_0+L_1s)$, which has both properties and lies below \eqref{eq:necessary} for every $s$ once $\delta=e^\rho-1$. Satisfying the necessary condition alone does not establish bounded distortion. We prove that this construction is sufficient next, and extend it to the loss-dependent scale.

\subsection{Admissible scales and the CS-SAM update}
\label{sec:cssam}

\Cref{thm:necessary-scaling} motivates a perturbation that divides the gradient by a local curvature upper bound. In this subsection we isolate the properties of such a bound that our analysis uses, state the update, and prove that it keeps the distortion bounded for any scale with these properties.

\begin{definition}[Admissible curvature scale]
    \label{def:admissible}
    Let $\kappa\ge0$. A function $M:\R^d\to(0,\infty)$ is \emph{$\kappa$-admissible} for $f$ if, for all $x,y\in\R^d$,
    \begin{align}
        \norm{\nabla^2 f(x)}\le M(x),
        \qquad
        M(y)\le M(x)\,e^{\kappa\norm{y-x}},
        \qquad
        \kappa\norm{\nabla f(x)}\le M(x).
        \label{eq:admissible}
    \end{align}
\end{definition}
Let us read the three conditions in \eqref{eq:admissible}. The first says that $M$ is a curvature scale, that is, it upper-bounds the local curvature everywhere. The second says that the scale varies at most exponentially in the distance travelled, so that a short move changes it by a bounded factor. The third ties the scale to the gradient norm; it is what makes a step proportional to $\nabla f(x)/M(x)$ short, so that it stays under control along the perturbation. Under \eqref{eq:assL}, the scale $M_L(x)=L_0+L_1\norm{\nabla f(x)}$ is $L_1$-admissible, so the concrete scale suggested by \Cref{thm:necessary-scaling} is an instance of this definition. The following proposition records this together with the corresponding facts for the scales of \eqref{eq:assSmooth} and \eqref{eq:assH}.

\begin{proposition}[Properties of the curvature scales]
    \label{prop:admissible}
    \leavevmode
    \begin{enumerate}[label=(\roman*), align=left, leftmargin=*, widest=ii]
        \item Under \eqref{eq:assSmooth}, the constant scale $M\equiv L$ is
        admissible with $\kappa=0$.

        \item Under \eqref{eq:assL}, the scale $M_L$ is admissible with
        $\kappa=L_1$.

        \item Under \eqref{eq:assH}, either $M_H\equiv0$ and $f$ is constant,
        or $M_H>0$ everywhere and the scale $M_H$ is admissible with
        $\kappa=\sqrt{H_1}$.

        \item Under \eqref{eq:assL}, for every $x\in\R^d$,
        \begin{align}
            L_0+L_1\norm{\nabla f(x)}
            \le\frac32L_0+\frac{3L_1^2}{2\,(2\log2-1)}\bigl(f(x)-f_\star\bigr).
            \label{eq:L-implies-H}
        \end{align}
        In particular, \eqref{eq:assL} implies \eqref{eq:assH} with the
        constants $H_0=3L_0/2$ and $H_1=3L_1^2/(2(2\log2-1))$.
    \end{enumerate}
\end{proposition}
Items (i)--(iii) say that the three curvature bounds in \eqref{eq:assSmooth}, \eqref{eq:assL} and \eqref{eq:assH} are admissible, with $\kappa=0$, $\kappa=L_1$, and $\kappa=\sqrt{H_1}$ respectively; the only exception is the trivial case of a constant $f$, for which $M_H$ vanishes and every point is a global minimizer. Item (iv) is an inclusion between the two general models: every objective satisfying \eqref{eq:assL} also satisfies \eqref{eq:assH}, with $H_0=3L_0/2$ and $H_1=3L_1^2/(2(2\log2-1))$. This inclusion also appears as Proposition~B.1 of \citet{alimisis2026warmup} and as Proposition~3 of \citet{vaswani2025armijo}, with different constants. 

\paragraph{The update.} Given an admissible curvature scale $M$, parameters $\rho\ge0$ and
$\gamma>0$, and a starting point $x_0$, we propose Curvature-Scaled SAM, which iterates
\begin{align}
    y_t=x_t+\frac{\rho}{M(x_t)}\,\nabla f(x_t),
    \qquad
    x_{t+1}=x_t-\frac{\gamma}{M(y_t)}\,\nabla f(y_t).
    \tag{CS-SAM}\label{eq:cssam}
\end{align}
Here, $\rho\ge0$ sets the perturbation and $\gamma>0$ the descent step. The perturbation step $\rho/M(x_t)$ and the descent step $\gamma/M(y_t)$ use the \emph{same} curvature scale, each evaluated at the point where the corresponding gradient is taken. If $\nabla f(x_t)=0$ the method stops at $x_t$. Two instances are of particular interest and are compared in \Cref{tab:scales}. With the gradient scale $M=M_L$, the perturbation step is $\rho/(L_0+L_1\norm{\nabla f(x_t)})$ and the descent step is $\gamma/(L_0+L_1\norm{\nabla f(y_t)})$; in particular, the perturbation multiplier is exactly the $\alpha(s)=\rho/(L_0+L_1s)$ permitted by \Cref{thm:necessary-scaling}. With the loss scale $M=M_H$, the perturbation step is $\rho/(H_0+H_1(f(x_t)-f_\star))$ and the descent step is $\gamma/(H_0+H_1(f(y_t)-f_\star))$. Note that in this case, the perturbation multiplier is \emph{not} covered by \Cref{thm:necessary-scaling}. Like \ref{eq:sam} and \ref{eq:usam}, \ref{eq:cssam} costs two gradient evaluations per iteration; the scale $M_L$ is computed from the gradient already in hand, and $M_H$ from the function value and $f_\star$. 

\medskip
\noindent
Let us highlight the degenerate case. For $\rho=0$ we have $y_t=x_t$, and \ref{eq:cssam} reduces to curvature-scaled gradient descent, that is, gradient descent with the curvature-scaled step $\gamma/M(x_t)$:
\begin{equation}
    x_{t+1}=x_t-\frac{\gamma}{M(x_t)}\,\nabla f(x_t).
    \tag{CS-GD}\label{eq:csgd}
\end{equation}
\ref{eq:cssam} is thus a generalization of \ref{eq:csgd}, whose instances at the three scales are listed in the last column of \Cref{tab:scales}, and in \Cref{sec:rates} we compare its guarantees with known gradient descent results.

\begin{table}[t]
    \centering
    \caption{The curvature scales associated with \eqref{eq:assSmooth}, \eqref{eq:assL} and \eqref{eq:assH} as instances of \ref{eq:cssam}. For each scale we list the assumption under which it is a curvature scale, the admissibility constant $\kappa$ of \Cref{prop:admissible}, the perturbation and descent steps of \ref{eq:cssam}, and the curvature-scaled gradient method obtained at $\rho=0$ (\ref{eq:csgd}, \Cref{sec:cssam}).}
    \label{tab:scales}
    \resizebox{\linewidth}{!}{%
    \begin{tabular}{l c c c c l}
    \toprule
    \textbf{Curvature scale $M(x)$} & \textbf{Assumption} & $\boldsymbol{\kappa}$ & \textbf{Perturbation step $\rho/M(x_t)$} & \textbf{Descent step $\gamma/M(y_t)$} & \textbf{\ref{eq:cssam} at $\rho=0$ (\ref{eq:csgd})} \\
    \midrule
    $L$ (constant) & \eqref{eq:assSmooth} & $0$ & $\rho/L$ & $\gamma/L$ & gradient descent with step $\gamma/L$ \\
    \addlinespace
    $M_L(x)=L_0+L_1\norm{\nabla f(x)}$ & \eqref{eq:assL} & $L_1$ & $\dfrac{\rho}{L_0+L_1\norm{\nabla f(x_t)}}$ & $\dfrac{\gamma}{L_0+L_1\norm{\nabla f(y_t)}}$ & \shortstack[l]{the gradient method of \citet{vankov2025optimizing}, \\ \citet{gorbunov2025methods}} \\
    \addlinespace
    $M_H(x)=H_0+H_1(f(x)-f_\star)$ & \eqref{eq:assH} & $\sqrt{H_1}$ & $\dfrac{\rho}{H_0+H_1(f(x_t)-f_\star)}$ & $\dfrac{\gamma}{H_0+H_1(f(y_t)-f_\star)}$ & \shortstack[l]{a loss-dependent step similar to\\ \citet{alimisis2026warmup}} \\
    \bottomrule
    \end{tabular}
    }
\end{table}

\medskip
\noindent
In contrast to \ref{eq:sam} and \ref{eq:usam}, the relative distortion of \ref{eq:cssam} is uniformly bounded for \emph{any} admissible curvature scale, as we now show.

\begin{theorem}
    \label{thm:relative-control}
    Let $M$ be $\kappa$-admissible, let $\rho\ge0$, $\gamma>0$, and run \ref{eq:cssam}. Then every iteration with $\nabla f(x_t)\ne0$ satisfies
    \begin{equation}
        e^{-\rho}M(x_t)\le M(y_t)\le e^\rho M(x_t),
        \qquad
        \frac{\norm{\nabla f(y_t)-\nabla f(x_t)}}{\norm{\nabla f(x_t)}}\le e^\rho-1.
    \end{equation}
    In particular, $D_{\mathrm{CS\text{-}SAM}}\le e^\rho-1$.
\end{theorem}
Let us highlight what \Cref{thm:relative-control} says. The curvature scale at the perturbed point stays within a factor $e^{\pm\rho}$ of the scale at the iterate, and the relative gradient distortion never exceeds $e^\rho-1$. This is exactly the uniform control that \eqref{eq:heuristic-bounds-smooth}, \eqref{eq:heuristic-bounds-l0l1}, and \Cref{prop:endpoint-failures} show \ref{eq:sam} and \ref{eq:usam} lack: the bound does not degenerate as $\norm{\nabla f(x_t)}\to0$, nor as $\norm{\nabla f(x_t)}\to\infty$. Moreover, the bound depends neither on the admissibility constant $\kappa$ nor on which admissible scale is used, so the constant scale, the $(L_0,L_1)$ scale, and the $(H_0,H_1)$ scale all inherit it. In particular, taking $\delta=e^\rho-1$ in \Cref{sec:distortion}, $D_{\mathrm{CS\text{-}SAM}}<1$ whenever $\rho<\log2$, which is a necessary condition for the rates of \Cref{sec:rates}: the condition $\chi(\rho,\gamma)>0$ of \eqref{eq:constants} under which they hold forces $\rho<\log2$.

\section{Convergence rates}
\label{sec:rates}

In this section, we present convergence rates for \ref{eq:cssam} run with the curvature scales $M_L$ of \eqref{eq:assL} and $M_H$ of \eqref{eq:assH}. We first fix the constants and the function classes that appear in the bounds. Then, \Cref{sec:L} states the rates for the gradient scale $M_L$ and \Cref{sec:H} those for the loss scale $M_H$, each as one theorem with a nonconvex, a convex, and a PL item. The proofs can be found in \Cref{app:L,app:H}; the auxiliary lemmas they rely on are collected in \Cref{app:lemmas}.

\paragraph{Constants.}
For $u\ge0$ let $\omega(u):=(e^u-1-u)/u^2$ for $u>0$ and $\omega(0):=1/2$. Define
\begin{align}
    \chi(\rho,\gamma):=2-e^{\rho}-\gamma\,\omega(\gamma)\,e^{3\rho},
    \qquad
    c_0:=1-\frac{\rho\,(e^{\rho}-1)}{2\log 2-1}
    -\frac{\gamma e^{3\rho}}{2\,(2\log 2-1)}.
    \label{eq:constants}
\end{align}
The condition $\chi(\rho,\gamma)>0$ forces $\rho<\log 2$. All results below assume $\chi(\rho,\gamma)>0$; the convex results assume in addition $c_0\ge0$. The two conditions restrict how large the perturbation and the step may be taken. In the degenerate case $\rho=0$, that is, for \ref{eq:csgd}, the constants reduce to $\gamma\chi(0,\gamma)=1+2\gamma-e^{\gamma}$ and $c_0=1-\gamma/(2(2\log2-1))$.

\paragraph{Function classes.}
We say that $f$ is \emph{convex with a minimizer} if $f$ is convex and attains $f_\star$ at some $x_\star$. We say that $f$ satisfies the \emph{$\mu$-PL condition} for some $\mu>0$ if
\begin{align}
    \tfrac12\norm{\nabla f(x)}^2\ge\mu\bigl(f(x)-f_\star\bigr)
    \qquad\forall x\in\R^d.
    \label{eq:PL}
\end{align}

\subsection{\texorpdfstring{$(L_0,L_1)$ curvature scale: $M=M_L$}{(L0,L1) curvature scale: M = M\_L}}
\label{sec:L}

Throughout this subsection $f$ satisfies \eqref{eq:assL} and \ref{eq:cssam} is run with the curvature scale $M=M_L$, so that the perturbation step is $\rho/(L_0+L_1\norm{\nabla f(x_t)})$ and the descent step is $\gamma/(L_0+L_1\norm{\nabla f(y_t)})$. We collect the three rates in one statement.

\begin{theorem}
    \label{thm:L-rates}
    Let $f$ satisfy \eqref{eq:assL}, let $\rho\ge0$ and $\gamma>0$ satisfy $\chi(\rho,\gamma)>0$, and let $(x_t)_{t\ge0}$ be generated by \ref{eq:cssam} with $M=M_L$. Then, for every $T\ge1$, the following hold:
    \begin{enumerate}[label=(\roman*)]
        \item
        \textit{(Nonconvex rate)}
        \begin{align}
            \min_{0\le t<T}\norm{\nabla f(x_t)}^2
            \le\left(\sqrt{\frac{e^{\rho}L_0\,\bigl(f(x_0)-f_\star\bigr)}{\gamma\chi(\rho,\gamma)\,T}}
            +\frac{e^{\rho}L_1\,\bigl(f(x_0)-f_\star\bigr)}{\gamma\chi(\rho,\gamma)\,T}\right)^{2}.
            \label{eq:L-nonconvex}
        \end{align}
        \item
        \textit{(Convex rate)} If in addition $f$ is convex with a minimizer $x_\star$ and $c_0\ge0$, then
        \begin{align}
            f(x_T)-f_\star
            \le\frac{e^{\rho}\Bigl(L_0\norm{x_0-x_\star}^2
            +L_1\norm{x_0-x_\star}\bigl(f(x_0)-f_\star\bigr)\Bigr)}
            {\gamma\chi(\rho,\gamma)\,T}.
            \label{eq:L-convex}
        \end{align}
        \item
        \textit{(PL rate)} If in addition $f$ satisfies the $\mu$-PL condition \eqref{eq:PL}, then
        \begin{align}
            f(x_T)-f_\star
            \le\left(1-\frac{2\mu\gamma e^{-\rho}\chi(\rho,\gamma)}
            {L_0+L_1\sqrt{2\mu\,\bigl(f(x_0)-f_\star\bigr)}}\right)^{T}
            \bigl(f(x_0)-f_\star\bigr).
            \label{eq:L-PL-a}
        \end{align}
    \end{enumerate}
\end{theorem}
Let us read the three items in turn. Item (i) shows that the best squared gradient norm decays as $\mathcal{O}(1/T)$: expanding the square, the leading term is $e^{\rho}L_0(f(x_0)-f_\star)/(\gamma\chi(\rho,\gamma)T)$ and depends on $L_0$ alone, while the gradient-dependent constant $L_1$ enters only through terms of lower order in $T$. Note that the bound requires nothing beyond \eqref{eq:assL} and $\chi(\rho,\gamma)>0$, that is, no bounded gradients and no bounded iterates, and that running the method needs nothing beyond the constants $L_0,L_1$ that define the scale. Item (ii) gives an $\mathcal{O}(1/T)$ bound on the gap of the last iterate for convex objectives, with a leading term proportional to $L_0\norm{x_0-x_\star}^2$ and a second term proportional to $L_1\norm{x_0-x_\star}(f(x_0)-f_\star)$. The additional condition $c_0\ge0$ is what keeps the iterates within the initial distance of $x_\star$ (\Cref{lem:fejer}), and it is the only place where the second constant of \eqref{eq:constants} enters. Item (iii) gives a linear rate under the PL condition.

\paragraph{Comparison with gradient methods at $\rho=0$.}
Let us now compare our guarantees with the known ones in the degenerate case. For $\rho=0$, \ref{eq:cssam} with $M=M_L$ is the gradient descent with step $\gamma/(L_0+L_1\norm{\nabla f(x_t)})$, and the constant $e^{\rho}/(\gamma\chi(\rho,\gamma))$ in items (i) and (ii) equals $1/(1+2\gamma-e^{\gamma})$, which is $1/(2\log2-1)\approx2.59$ at $\gamma=\log2$. This is the method analyzed by \citet{vankov2025optimizing} and \citet{gorbunov2025methods}, and \Cref{thm:L-rates}(i)--(ii) give the same rates as Theorems~3.1 and 3.2 of \citet{vankov2025optimizing}: the best squared gradient norm decays as $1/T$ with the $L_1$ term of lower order, and the gap decays as $1/T$ with an $L_1\norm{x_0-x_\star}$ term. Our analysis therefore recovers the known asymptotic rate orders in $T$ of the curvature-scaled gradient method as the special case $\rho=0$, and extends them to the SAM setting. To the best of our knowledge, \Cref{thm:L-rates} provides the first convergence rates for a SAM-type method under the $(L_0,L_1)$ curvature scale.

\subsection{\texorpdfstring{$(H_0,H_1)$ curvature scale: $M=M_H$}{(H0,H1) curvature scale: M = M\_H}}
\label{sec:H}

Throughout this subsection $f$ satisfies \eqref{eq:assH} and is nonconstant, so that $M_H>0$ everywhere by \Cref{prop:admissible}(iii), and \ref{eq:cssam} is run with the curvature scale $M=M_H$: the perturbation step is $\rho/(H_0+H_1(f(x_t)-f_\star))$ and the descent step is $\gamma/(H_0+H_1(f(y_t)-f_\star))$. Note that the loss scale is evaluated relative to $f_\star$, which the method therefore needs, in contrast to the gradient scale of \Cref{sec:L}.

\begin{theorem}
    \label{thm:H-rates}
    Let $f$ satisfy \eqref{eq:assH} and be nonconstant, let $\rho\ge0$ and $\gamma>0$ satisfy $\chi(\rho,\gamma)>0$, and let $(x_t)_{t\ge0}$ be generated by \ref{eq:cssam} with $M=M_H$. Then, for every $T\ge1$, the following hold:
    \begin{enumerate}[label=(\roman*)]
        \item
        \textit{(Nonconvex rate)}
        \begin{align}
            \min_{0\le t<T}\norm{\nabla f(x_t)}^2
            \le\frac{e^{\rho}\bigl(H_0+H_1(f(x_0)-f_\star)\bigr)\bigl(f(x_0)-f_\star\bigr)}
            {\gamma\chi(\rho,\gamma)\,T}.
            \label{eq:H-nonconvex}
        \end{align}
        \item
        \textit{(Convex rate)} If in addition $f$ is convex with a minimizer $x_\star$ and $c_0\ge0$, then
        \begin{align}
            f(x_T)-f_\star
            \le\frac{e^{\rho}\bigl(H_0+H_1(f(x_0)-f_\star)\bigr)\norm{x_0-x_\star}^2}
            {\gamma\chi(\rho,\gamma)\,T}.
            \label{eq:H-convex}
        \end{align}
        \item
        \textit{(PL rate)} If in addition $f$ satisfies the $\mu$-PL condition \eqref{eq:PL}, then
        \begin{align}
            f(x_T)-f_\star
            \le\left(1-\frac{2\mu\gamma e^{-\rho}\chi(\rho,\gamma)}
            {H_0+H_1\bigl(f(x_0)-f_\star\bigr)}\right)^{T}
            \bigl(f(x_0)-f_\star\bigr).
            \label{eq:H-PL-a}
        \end{align}
    \end{enumerate}
\end{theorem}
The three items have the same shape as in \Cref{thm:L-rates}, with one difference worth pointing out: the dependence on the curvature parameters is different. Item (i) bounds the best squared gradient norm by $\mathcal{O}(1/T)$ with a constant proportional to $(H_0+H_1(f(x_0)-f_\star))(f(x_0)-f_\star)$. Thus, unlike the $L_1$ contribution in \Cref{thm:L-rates}(i), the $H_1$ contribution occurs at the leading $1/T$ order. Item (ii) gives an $\mathcal{O}(1/T)$ bound on the last-iterate gap for convex objectives, proportional to $(H_0+H_1(f(x_0)-f_\star))\norm{x_0-x_\star}^2$, again under $c_0\ge0$, which bounds the iterates as in \Cref{sec:L}. Item (iii) gives a linear rate under the PL condition, with contraction factor $1-2\mu\gamma e^{-\rho}\chi(\rho,\gamma)/(H_0+H_1(f(x_0)-f_\star))$. 

\paragraph{Comparison with gradient methods at $\rho=0$.}
For $\rho=0$, \ref{eq:cssam} with $M=M_H$ is gradient descent with the step $\gamma/(H_0+H_1(f(x_t)-f_\star))$, and, as in \Cref{sec:L}, the constant $e^{\rho}/(\gamma\chi(\rho,\gamma))$ in items (i) and (ii) equals $1/(1+2\gamma-e^{\gamma})$, which is $1/(2\log2-1)\approx2.59$ at $\gamma=\log2$. This kind of loss-dependent step was used by \citet{alimisis2026warmup}, who took $1/(10H_0+20H_1(f(x_t)-f_\star))$, and \Cref{thm:H-rates}(ii)--(iii) give the same rates as their Theorems~4.2 and 4.3, namely a $1/T$ gap in the convex case and a linear rate under the PL condition, with a larger step at $\gamma=\log2$: their denominator coefficients $10$ and $20$ are both replaced by $1/\log2\approx1.44$. At $\rho=0$ our guarantees thus match the known rates for gradient descent with a loss-dependent step and allow larger step sizes, and they extend to any positive perturbation with $\chi(\rho,\gamma)>0$. To the best of our knowledge, \Cref{thm:H-rates} provides the first convergence rates for a SAM-type method under the $(H_0,H_1)$ curvature scale. 

\begin{remark}
    As we show in \Cref{prop:admissible}(iv), every objective satisfying \eqref{eq:assL} also satisfies \eqref{eq:assH} with $H_0=3L_0/2$ and $H_1=3L_1^2/(2(2\log2-1))$, so one could ask why we treat the two scales separately. Note that \ref{eq:cssam} with $M=M_L$ and \ref{eq:cssam} with $M=M_H$ are the same algorithm run with different steps: while one assumption implies the other, the perturbation step and the descent step are scaled by $M_L$ in one case and by $M_H$ in the other. \Cref{thm:L-rates,thm:H-rates} are therefore guarantees for two different step choices. 
\end{remark}

\section{Numerical experiments}
\label{sec:experiments}

In this section, we first verify the results of \Cref{sec:failures} on the two one-dimensional objectives of \Cref{prop:endpoint-failures} (\Cref{sec:verification}), and we then test \ref{eq:cssam} on the training of deep neural networks against tuned \ref{eq:sam} and \ref{eq:usam} baselines (\Cref{sec:dnn}). The complete protocols are given in \Cref{app:exps}.

\subsection{Verification of \texorpdfstring{\Cref{sec:failures}}{Section 2}}
\label{sec:verification}

The experiments of this subsection are deterministic and one-dimensional, with exact function values and gradients. We use the quadratic $f(x)=x^2/2$ and the quartic $f(x)=x^4/4$ of \Cref{prop:endpoint-failures}. The two scales coincide on the quadratic, $M_L=M_H\equiv1$, while on the quartic $M_L(x)=4+|x|^3$ and $M_H(x)=3+3x^4/4$, that is, the constants $(L_0,L_1)=(4,1)$ of \Cref{prop:endpoint-failures}(ii) and $(H_0,H_1)=(3,3)$; both bound the curvature $3x^2$ at every $x$ (\Cref{app:verification-details}). Throughout, \ref{eq:cssam} is run with $\rho=\gamma=1/8$, for which $\chi(\rho,\gamma)>0$ and $c_0\ge0$. The results are collected in \Cref{fig:verification}.

\paragraph{Two-cycle and divergence.}
We first run the three methods on the two objectives. On the quadratic, \ref{eq:sam} with $r=1$ and $\eta=0.2$ is started at $x_0=3$; by \Cref{prop:endpoint-failures}(i), its iterates approach the two-cycle $\pm a$ with $a=\eta r/(2-\eta)=1/9$. On the quartic, \ref{eq:usam} with $\rho=1/8$ and $\eta=1/32$ is started at $x_0=4$, which exceeds the threshold $(2/(\eta\rho^3))^{1/8}\approx3.67$ of \Cref{prop:endpoint-failures}(ii). \Cref{fig:endpoint-quadratic,fig:endpoint-quartic} show the iterates. We observe that \ref{eq:sam} oscillates between $\pm1/9$ (dashed lines) and does not converge, and that \ref{eq:usam} changes sign and grows at every step, reaching $|x_2|\approx10^{11}$ after two updates, at which point we stop the run. In contrast, \ref{eq:cssam} converges to the minimizer on both objectives, with either scale.

\paragraph{Uniform distortion control.}
We next evaluate the relative gradient distortion \eqref{eq:distortion} of each perturbation as a function of the gradient norm, no outer step enters this measurement. Here \ref{eq:sam} uses $r=1/8$ on the quadratic and \ref{eq:usam} uses $\rho=1/8$ on the quartic, so that all three perturbations share the numerical parameter $1/8$. \Cref{fig:distortion-quadratic,fig:distortion-quartic} show the distortion together with the bound $e^{\rho}-1\approx0.133$ of \Cref{thm:relative-control}. We observe the two failures of \Cref{sec:distortion}: the distortion of \ref{eq:sam} equals $r/\norm{\nabla f(x)}$ and grows without bound as the gradient vanishes, while that of \ref{eq:usam} equals $(1+\rho x^2)^3-1$ and grows without bound as the gradient grows. The distortion of \ref{eq:cssam} stays below $e^{\rho}-1$ at every gradient norm and for both scales, as \Cref{thm:relative-control} guarantees. Note that on the quadratic it equals $\rho$ exactly, and on the quartic it peaks below the bound and decays at both ends.

\begin{figure}[t]
    \centering
    \begin{subfigure}[t]{0.24\linewidth}
        \includegraphics[width=\linewidth]{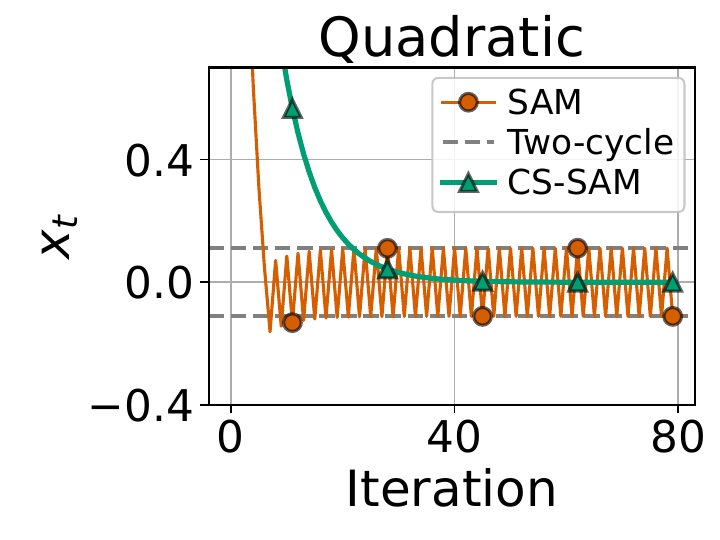}
        \caption{Quadratic: iterates}\label{fig:endpoint-quadratic}
    \end{subfigure}\hfill
    \begin{subfigure}[t]{0.24\linewidth}
        \includegraphics[width=\linewidth]{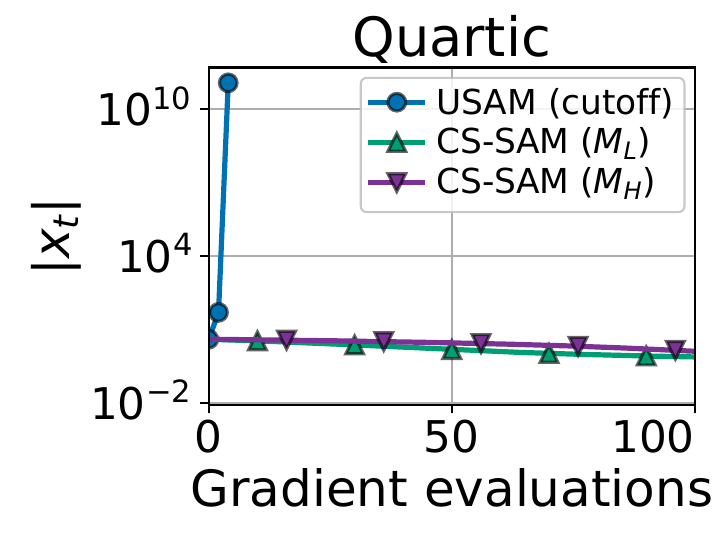}
        \caption{Quartic: $|x_t|$}\label{fig:endpoint-quartic}
    \end{subfigure}\hfill
    \begin{subfigure}[t]{0.24\linewidth}
        \includegraphics[width=\linewidth]{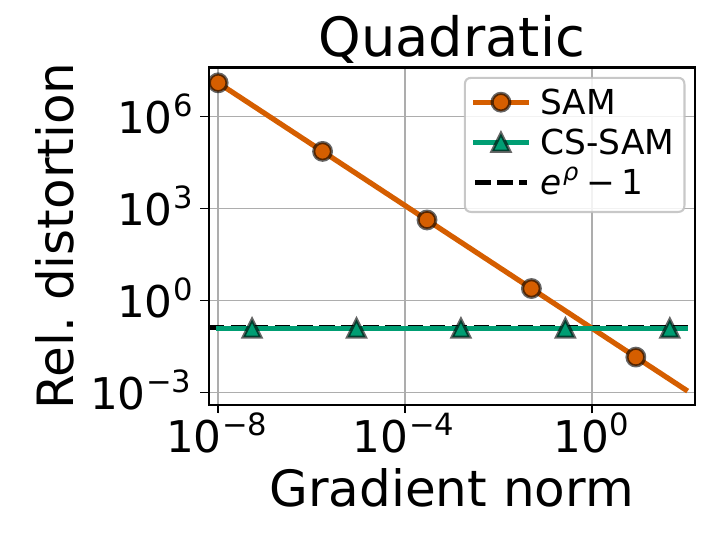}
        \caption{Quadratic: distortion}\label{fig:distortion-quadratic}
    \end{subfigure}\hfill
    \begin{subfigure}[t]{0.24\linewidth}
        \includegraphics[width=\linewidth]{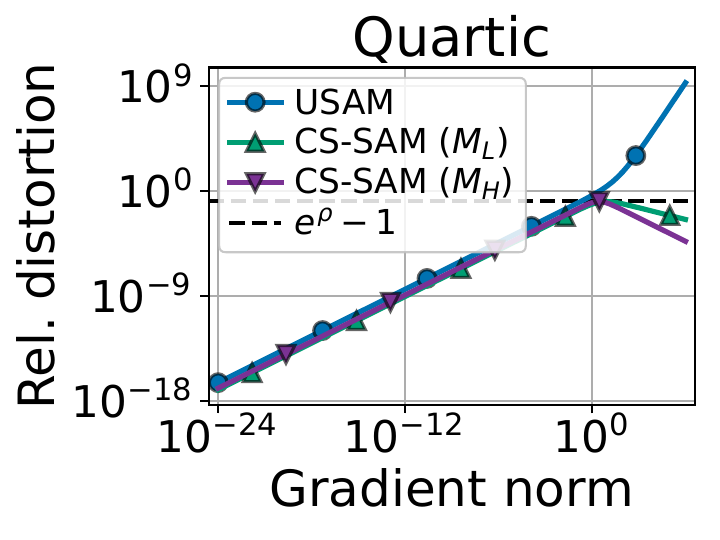}
        \caption{Quartic: distortion}\label{fig:distortion-quartic}
    \end{subfigure}
    \caption{Verification of \Cref{sec:failures} on the objectives of \Cref{prop:endpoint-failures}. (a) Iterates of \ref{eq:sam} on $f(x)=x^2/2$ against the predicted two-cycle (dashed), with \ref{eq:cssam}. (b) $|x_t|$ of \ref{eq:usam} on $f(x)=x^4/4$ against the number of gradient evaluations, with \ref{eq:cssam} for both scales. (c)--(d) Relative gradient distortion \eqref{eq:distortion} against the gradient norm and the bound $e^{\rho}-1$ of \Cref{thm:relative-control}.}
    \label{fig:verification}
\end{figure}

\subsection{DNN experiments}
\label{sec:dnn}

In this subsection, we test \ref{eq:cssam} on the training of deep neural networks and compare it with tuned \ref{eq:sam} and \ref{eq:usam} baselines. We go beyond the deterministic setting of \Cref{sec:verification} and of our guarantees: the networks are trained with minibatch gradients, and the curvature scales are evaluated on minibatch quantities. The experiments are designed to test whether the two scales are competitive in practice. The complete grids and the per-run configurations are given in \Cref{app:dnn-details}.

\medskip
\noindent
We train ResNet-20 and ResNet-32 with batch normalization on CIFAR-10 and CIFAR-100 for 200 epochs with batch size 128, random-crop and horizontal-flip augmentation without cutout, and the cross-entropy loss. Momentum and weight decay are $0$ for every method, so that the comparison isolates the update rules. In the minibatch setting the $M_L$ \emph{scale} is $L_0+L_1\norm{\nabla f_i(x)}$, and the $M_H$ \emph{scale} is $H_0+H_1[f_i(x)-\ell_i^\star]_+$, where $\ell_i^\star$ is a lower bound of $f_i$ and $[z]_+=\max\{z,0\}$. We set $\ell_i^\star=0$ in all our expriments. We compare against \ref{eq:sam} and \ref{eq:usam}, each with a constant and with a cosine annealing schedule, tuned on a $4\times4$ grid of learning rate $\times$ radius (respectively multiplier) for every model--dataset cell. For \ref{eq:cssam} we tune the $M_L$ scale over $\rho\in\{0.2,0.4,0.8,1.6\}$, $\gamma\in\{0.2,0.4,0.8\}$, $L_0\in\{0.02,0.5,2\}$, and the $M_H$ scale over $\rho,\gamma\in\{0.3,1,3\}$, $H_1\in\{1,3,10,30\}$. We run each experiment for 200 epochs. \Cref{tab:cifar} reports, for each method and cell, the mean and standard deviation over the three seeds of the best test accuracy. We observe that both curvature scales are competitive: CS-SAM with $M_H$ achieves the highest mean best-test accuracy in all four settings, while CS-SAM with $M_L$ beats tuned SAM and USAM in three out of four settings. 

\begin{table}[t]
    \centering
    \caption{Best test accuracy (\%, maximum over 200 epochs), mean $\pm$ standard deviation over 3 seeds, for ResNet-20/32 trained on CIFAR-10/100. Bold marks the best mean in each column.}
    \label{tab:cifar}
    \resizebox{\linewidth}{!}{%
    \begin{tabular}{l cccc}
    \toprule
    Method & ResNet-20 / CIFAR-10 & ResNet-32 / CIFAR-10 & ResNet-20 / CIFAR-100 & ResNet-32 / CIFAR-100 \\
    \midrule
    \ref{eq:sam}, constant lr & $90.60_{\pm 0.20}$ & $91.24_{\pm 0.08}$ & $64.66_{\pm 0.61}$ & $65.55_{\pm 0.62}$ \\
    \addlinespace
    \ref{eq:sam}, cosine & $90.49_{\pm 0.16}$ & $91.16_{\pm 0.23}$ & $65.32_{\pm 0.21}$ & $66.07_{\pm 0.69}$ \\
    \addlinespace
    \ref{eq:usam}, constant lr & $90.15_{\pm 0.30}$ & $90.91_{\pm 0.11}$ & $64.73_{\pm 0.12}$ & $65.61_{\pm 0.40}$ \\
    \addlinespace
    \ref{eq:usam}, cosine & $90.13_{\pm 0.41}$ & $90.84_{\pm 0.19}$ & $65.33_{\pm 0.13}$ & $66.77_{\pm 0.34}$ \\
    \addlinespace
    \ref{eq:cssam} with $M_L$ & $90.87_{\pm 0.30}$ & $91.72_{\pm 0.33}$ & $65.18_{\pm 0.42}$ & $66.80_{\pm 0.59}$ \\
    \addlinespace
    \ref{eq:cssam} with $M_H$ & $\mathbf{91.01_{\pm 0.06}}$ & $\mathbf{91.73_{\pm 0.22}}$ & $\mathbf{65.68_{\pm 0.34}}$ & $\mathbf{67.06_{\pm 0.44}}$ \\
    \bottomrule
    \end{tabular}
    }
\end{table}

\section{Conclusion}
\label{sec:conc}

In this work, we studied sharpness-aware minimization under generalized smoothness through the relative gradient distortion. We showed that the distortion of \ref{eq:sam} and \ref{eq:usam} can become unbounded at opposite ends of the gradient scale, derived a necessary condition for uniform distortion control, and introduced \ref{eq:cssam}, which scales the perturbation and the descent step by the \emph{same} admissible curvature scale and keeps the distortion bounded. Under $(L_0,L_1)$- and $(H_0,H_1)$-smoothness, we established $\mathcal{O}(1/T)$ rates for the best squared gradient norm on nonconvex objectives and for the last-iterate gap on convex objectives, and linear convergence under the PL condition, recovering the known gradient descent rates for $\rho=0$. Our experiments confirmed the two failures and the distortion bound on one-dimensional objectives, and showed that both curvature scales are competitive with tuned \ref{eq:sam} and \ref{eq:usam} on ResNet-20/32 and CIFAR-10/100, beyond the deterministic setting of our theory. Future directions include extending the analysis to the stochastic setting, in which the curvature scales are evaluated on minibatch quantities as in our experiments, removing the dependence of the loss scale on $f_\star$, and combining \ref{eq:cssam} with momentum and weight decay in larger-scale training.

\section*{AI use statement}
We used generative AI tools to suggest proof ideas, which the authors checked and wrote up, and to streamline and edit the wording of some paragraphs and sentences. We take responsibility for the final content of this work.

\section*{Acknowledgments}
The authors’ contribution to this work was supported by NSF CAREER 2542902, and AFOSR YIP FA9550-26-1-B353.

\bibliographystyle{plainnat}
\bibliography{refs}

\newpage
\appendix
\crefalias{section}{appendix}
\crefalias{subsection}{appendix}

\newenvironment{restatement}[1]{%
    \begin{mdframed}[skipabove=8pt,skipbelow=8pt,hidealllines=true,backgroundcolor=shadecolor,innerleftmargin=4pt,innerrightmargin=4pt]%
    \noindent{\bfseries\Cref{#1}} (restated){\bfseries.}\hspace{0.5em}\ignorespaces}{\end{mdframed}}

\part*{Supplementary Material}

The appendix is organized as follows. In \Cref{app:related}, we give a more detailed discussion of related work. In \Cref{app:lemmas}, we collect the auxiliary lemmas on admissible curvature scales and on a single iteration of \ref{eq:cssam} that all proofs rely on. \Cref{app:endpoint} presents the proofs for \Cref{sec:failures}: the verification of the distortion bounds \eqref{eq:heuristic-bounds-smooth} and \eqref{eq:heuristic-bounds-l0l1}, and the proofs of \Cref{prop:endpoint-failures}, \Cref{thm:necessary-scaling}, \Cref{prop:admissible}, and \Cref{thm:relative-control}. \Cref{app:rates} presents the proofs for \Cref{sec:rates}, with \Cref{app:L} for the $(L_0,L_1)$ curvature scale (\Cref{thm:L-rates}) and \Cref{app:H} for the $(H_0,H_1)$ curvature scale (\Cref{thm:H-rates}). Finally, \Cref{app:exps} gives the experimental details for \Cref{sec:experiments}.

{\small\tableofcontents}

\newpage
\section{Further Related Work}
\label{app:related}

In this section, we complement the discussion of related work in \Cref{sec:intro} with a more detailed survey. We proceed from the broader literature on the dynamics of \ref{eq:sam} and on adaptive perturbation rules to the lines closest to our work: relative control of the perturbation-induced gradient error, optimization under the curvature models \eqref{eq:assL} and \eqref{eq:assH}, and relative gradient models. We close each paragraph with a brief comparison with our results.

\paragraph{SAM, normalization, and optimization dynamics.}
Sharpness-Aware Minimization replaces the gradient at the current iterate with a gradient evaluated after a parameter perturbation, motivated by minimizing the loss over a fixed local neighborhood \citep{foret2021sam}. Subsequent work has shown that its behavior cannot be understood solely as ordinary descent on a smooth objective. \citet{andriushchenko2022understanding} studied the implicit bias of \ref{eq:sam} and the role of stochastic sharpness, while \citet{dai2023normalization} identified stabilization and continued motion along manifolds of minima as important effects of the gradient normalization. \citet{bartlett2022dynamics} showed that fixed-radius \ref{eq:sam} typically approaches a two-cycle on convex quadratics and that analogous oscillations can induce drift toward regions with smaller Hessian spectral norm on nonquadratic landscapes. \citet{si2023practical} established unavoidable stationarity neighborhoods for practical fixed-radius \ref{eq:sam} in several settings, and \citet{park2026hallucinate} showed that \ref{eq:sam} can converge to points at which the perturbed-point gradient vanishes even though the original gradient does not. Our quadratic construction in \Cref{prop:endpoint-failures}(i) should therefore be viewed as an exact illustration of established late-stage phenomena. Our new point is that the low-gradient failure of \ref{eq:sam} and the high-gradient amplification of \ref{eq:usam} are complementary consequences of expressing the perturbation budget in inappropriate units, which \Cref{thm:necessary-scaling} makes precise.

\paragraph{Unified and adaptive perturbation rules.}
Unified SAM \citep{oikonomou2025general} interpolates between \ref{eq:sam} and \ref{eq:usam} through a fixed normalization rule and provides deterministic and stochastic convergence guarantees under several sampling and step-size choices. Other methods adapt the perturbation: ASAM \citep{kwon2021asam} rescales the parameter-space neighborhood to obtain scale invariance, whereas Fisher SAM \citep{kim2022fishersam} replaces the Euclidean ball by a Fisher-information geometry. LETS \citep{wang2024lets} learns the perturbation radius through a bilevel training--validation objective, LightSAM \citep{cheng2025lightsam} adapts the perturbation radius and the outer learning rate using AdaGrad- and Adam-style statistics, and LE-SAM \citep{wang2026lesam} replaces the fixed parameter-space radius by a fixed loss-space budget. More recent methods adapt the radius from the observed Taylor-approximation error \citep{deng2026tlcsam} or redistribute a fixed global radius across parameter blocks using moving averages of blockwise gradient energy \citep{huang2026gearsam}. Polyak-type SAM schedulers instead adapt the outer step from function values under classical smoothness \citep{oikonomou2026polyak}. In contrast, \ref{eq:cssam} adapts the perturbation through one scalar quantity, an upper bound on the local curvature (\Cref{def:admissible}), rather than through validation performance, optimizer statistics, parameter scale, blockwise gradient energy, or a loss-space budget. Moreover, the scaling of the perturbation is derived rather than chosen: it is motivated by the necessary condition of \Cref{thm:necessary-scaling} for uniformly bounded relative gradient distortion, and \Cref{thm:relative-control} shows that it meets this requirement for every admissible scale. Finally, the same curvature scale sets the outer step, so that a single quantity controls both the perturbation and the descent step.

\paragraph{Relative-error control and exact stationarity.}
Inexact-gradient analysis provides a natural framework for studying gradients evaluated at perturbed points. \citet{khanh2024fundamental} used this perspective to establish convergence results for several normalized and unnormalized SAM variants, including conditions that control the perturbation-induced gradient error. \citet{adly2026stationarity} analyzed the deterministic family
\begin{equation*}
    y_t=x_t+\rho\,\frac{\nabla f(x_t)}{\norm{\nabla f(x_t)}^{\alpha}},
    \qquad 0\le\alpha\le1,
\end{equation*}
under classical smoothness \eqref{eq:assSmooth}. They identified stationarity floors for constant-parameter rules and proposed a clipped perturbation that agrees with the original rule away from stationarity but becomes gradient-proportional near stationary points, yielding $\norm{\nabla f(x_t)}\to0$. Relative control of the gradient error is therefore not new by itself. Our contribution is to derive the scaling required to enforce such control uniformly under gradient-dependent curvature: by \Cref{thm:necessary-scaling}, every perturbation in the gradient direction with uniformly bounded distortion over the class \eqref{eq:assL} must be \ref{eq:usam}-like near stationarity and \ref{eq:sam}-like at large gradients. The curvature-scaled perturbation of \ref{eq:cssam} realizes both regimes and couples the resulting inner perturbation to a curvature-scaled outer step.

\paragraph{Optimization beyond uniform smoothness.}
The $(L_0,L_1)$-smoothness model \eqref{eq:assL} was introduced to describe objectives whose local curvature can grow with the gradient norm and to explain the effectiveness of clipping and normalized gradients \citep{zhang2020clipping}. Subsequent work developed sharper clipping, acceleration, normalization, and adaptive methods for this class \citep{zhang2020improved,crawshaw2022robustness,chen2023generalized,koloskova2023revisiting,li2023convex,gorbunov2025methods,vankov2025optimizing}. Curvature-scaled ordinary gradient steps are therefore established machinery rather than a contribution of this paper. The same model has also been extended to variational inequalities and root-finding problems, with adaptive-stepsize and extragradient-type methods \citep{vankov2024generalized, yoon2026polyak, choudhury2025extragradient,greenberg2026extragradient}. The loss-dependent condition \eqref{eq:assH} has likewise been used to explain learning-rate warm-up \citep{alimisis2026warmup} and to analyze line search, steepest descent, and adaptive methods \citep{vaswani2025armijo,vaswani2026nonuniform}. Our contribution is to use the same machinery for the perturbation: in \ref{eq:cssam}, the curvature scale that sets the descent step also sets the perturbation, and \Cref{thm:relative-control} shows that this keeps the relative gradient distortion uniformly bounded. The resulting rates have the same asymptotic orders in $T$ as those of curvature-scaled gradient descent, with different dependence on the curvature parameters and initial conditions, and they recover the known gradient-descent rates at $\rho=0$ (\Cref{sec:rates}). 

\paragraph{Relative gradient models and curvature geometry.}
Gradient-normalized smoothness characterizes the neighborhood in which the gradient admits a reliable relative approximation, and uses it to globalize methods based on approximate Hessians \citep{semenov2026gns}. We specialize this relative-locality principle to a first-order sharpness-aware perturbation, with an explicit rule based on a scalar curvature scale and a necessity result for perturbations in the gradient direction (\Cref{thm:necessary-scaling}). The exponential local models of \eqref{eq:admissible} are also reminiscent of generalized self-concordance, where controlled curvature growth replaces a uniform Hessian bound \citep{sun2019generalized}; unlike that convex, Newton-oriented framework, our analysis certifies a two-gradient SAM update on nonconvex objectives.

\paragraph{Geometry, generalization, and scope.}
Parameter-space sharpness is sensitive to neural-network reparameterization, and empirical evaluations show that no single sharpness measure universally predicts generalization \citep{kwon2021asam,andriushchenko2023modern}. Fisher SAM and $\mu$P$^2$ further illustrate that information-geometric or layerwise perturbation scaling can be important in modern networks \citep{kim2022fishersam,haas2024mup2}. Moreover, late-stage SAM can be especially effective for selecting flatter minima \citep{zhou2025late}, while standard parameter-space SAM can optimize spurious notions of sharpness in language modeling \citep{singh2025spurious}. \ref{eq:cssam} should therefore be viewed as an optimization method rather than a parameterization-invariant theory of generalization. The present analysis is deterministic, assumes exact gradients and function values, and uses a global Euclidean scalar curvature scale.

\newpage
\section{Auxiliary lemmas}
\label{app:lemmas}

In this section, we collect the lemmas used in \Cref{app:endpoint,app:rates}. All of them are stated for a general $\kappa$-admissible scale $M$ (\Cref{def:admissible}), with three exceptions: \Cref{lem:gronwall} is the differential form of Gr\"onwall's inequality, \Cref{lem:selfbound,lem:zeroscale} concern the loss scale $M_H$ of \eqref{eq:assH} and \Cref{lem:recursion} is an elementary fact about sequences. 

\paragraph{Conventions for the proofs.}
The constants $\chi(\rho,\gamma)$ and $c_0$ are defined in \eqref{eq:constants}. If the method stops at some $x_t$, then $\nabla f(x_t)=0$: the best gradient norm is then zero, and under convexity or the PL condition $x_t$ is a global minimizer, so every bound below holds trivially from that index on. We therefore assume in all proofs that $\nabla f(x_t)\neq0$ for every $t$. Note that $\chi(\rho,\gamma)>0$ implies $2-e^{\rho}>\gamma\omega(\gamma)e^{3\rho}>0$, hence $\rho<\log2$ and $\gamma e^{-\rho}\chi(\rho,\gamma)>0$. Finally, recall that $\omega(u)=\sum_{k\ge0}u^k/(k+2)!$, so $\omega$ is increasing on $[0,\infty)$, and that
\begin{align}
    u+u^2\omega(u)=e^u-1,\qquad u-u^2\omega(u)=1+2u-e^u\qquad(u\ge0).
    \label{eq:omega-identities}
\end{align}

\medskip
\noindent
We begin with the differential form of Gr\"onwall's inequality \citep{gronwall1919note}, stated for absolutely continuous functions because the gradient scale $M_L$ is only Lipschitz along a segment. We include the short proof for completeness.

\begin{lemma}[Gr\"onwall's inequality; see \citealp{gronwall1919note}]
    \label{lem:gronwall}
    Let $c\ge0$ and let $g:[0,1]\to[0,\infty)$ be absolutely continuous with
    $g'(s)\le c\,g(s)$ for almost every $s\in[0,1]$. Then $g(s)\le g(0)e^{cs}$
    for every $s\in[0,1]$.
\end{lemma}

\begin{proof}
    The function $h(s):=e^{-cs}g(s)$ is absolutely continuous on $[0,1]$, with
    $h'(s)=e^{-cs}\bigl(g'(s)-c\,g(s)\bigr)\le0$ for almost every $s$. Hence
    $h(s)=h(0)+\int_0^s h'(u)\,du\le h(0)=g(0)$, that is $g(s)\le g(0)e^{cs}$.
\end{proof}

\begin{lemma}
    \label{lem:local}
    Let $M$ be $\kappa$-admissible for $f$. Then, for all $x,d\in\R^d$,
    \begin{align}
        \begin{gathered}
            \norm{\nabla f(x+d)-\nabla f(x)}
            \le M(x)\,\frac{e^{\kappa\norm d}-1}{\kappa\norm d}\,\norm d,\\
            \bigl|f(x+d)-f(x)-\ip{\nabla f(x)}{d}\bigr|
            \le M(x)\,\omega(\kappa\norm d)\,\norm d^2,
        \end{gathered}
        \label{eq:local}
    \end{align}
    where the quotient is interpreted as $1$ when $\kappa\norm d=0$.
\end{lemma}

\begin{proof}
    By the fundamental theorem of calculus and the first two admissibility
    conditions,
    \begin{align*}
        \norm{\nabla f(x+d)-\nabla f(x)}
        &\le\int_0^1\norm{\nabla^2f(x+sd)}\norm d\,ds\\
        &\le\norm d\int_0^1 M(x)e^{\kappa s\norm d}\,ds
        =M(x)\norm d\,\frac{e^{\kappa\norm d}-1}{\kappa\norm d}.
    \end{align*}
    For the second inequality, write
    \begin{align*}
        f(x+d)-f(x)-\ip{\nabla f(x)}{d}=\int_0^1\ip{\nabla f(x+sd)-\nabla f(x)}{d}\,ds
    \end{align*}
    and apply the first inequality with $d=sd$, which gives
    \begin{align*}
        \norm{\nabla f(x+sd)-\nabla f(x)}\le M(x)\norm d\,\frac{e^{s\kappa\norm d}-1}{\kappa\norm d}.
    \end{align*}
    Integrating,
    \begin{align*}
        \bigl|f(x+d)-f(x)-\ip{\nabla f(x)}{d}\bigr|
        &\le M(x)\norm d^2\int_0^1\frac{e^{s\kappa\norm d}-1}{\kappa\norm d}\,ds\\
        &=M(x)\norm d^2\,\frac{e^{\kappa\norm d}-1-\kappa\norm d}{(\kappa\norm d)^2}
        =M(x)\,\omega(\kappa\norm d)\norm d^2 .\qedhere
    \end{align*}
\end{proof}

\begin{lemma}
    \label{lem:gap}
    Let $M$ be $\kappa$-admissible for $f$. Then, for every $x\in\R^d$,
    \begin{align}
        f(x)-f_\star\ge(2\log2-1)\,\frac{\norm{\nabla f(x)}^2}{M(x)} .
        \label{eq:gap}
    \end{align}
\end{lemma}

\begin{proof}
    Fix $\alpha\ge0$ and consider the point $x-\alpha\nabla f(x)/M(x)$. The third
    admissibility condition gives
    $\kappa\norm{\alpha\nabla f(x)/M(x)}=\alpha\kappa\norm{\nabla f(x)}/M(x)\le\alpha$,
    so by \Cref{lem:local}, the monotonicity of $\omega$, and
    \eqref{eq:omega-identities},
    \begin{align*}
        f_\star\le f\Bigl(x-\alpha\frac{\nabla f(x)}{M(x)}\Bigr)
        &\le f(x)-\alpha\frac{\norm{\nabla f(x)}^2}{M(x)}
        +M(x)\,\omega(\alpha)\,\alpha^2\frac{\norm{\nabla f(x)}^2}{M(x)^2}\\
        &=f(x)-\bigl(\alpha-\alpha^2\omega(\alpha)\bigr)\frac{\norm{\nabla f(x)}^2}{M(x)}
        =f(x)-\bigl(1+2\alpha-e^{\alpha}\bigr)\frac{\norm{\nabla f(x)}^2}{M(x)} .
    \end{align*}
    Thus we get
    $f(x)-f_\star\ge(1+2\alpha-e^{\alpha})\norm{\nabla f(x)}^2/M(x)$ for every
    $\alpha\ge0$. The coefficient is maximized at $\alpha=\log2$, where it equals
    $2\log2-1$.
\end{proof}

\begin{lemma}
    \label{lem:geometry}
    Let $M$ be $\kappa$-admissible for $f$, let $\rho\ge0$, and let $x_t$ be an iterate of \ref{eq:cssam}. Then
    \begin{gather}
        e^{-\rho}M(x_t)\le M(y_t)\le e^{\rho}M(x_t),\qquad
        \norm{\nabla f(y_t)-\nabla f(x_t)}\le(e^{\rho}-1)\norm{\nabla f(x_t)},
        \label{eq:geometry1}\\
        \norm{\nabla f(y_t)}\le e^{\rho}\norm{\nabla f(x_t)},\qquad
        \ip{\nabla f(x_t)}{\nabla f(y_t)}\ge(2-e^{\rho})\norm{\nabla f(x_t)}^2 .
        \label{eq:geometry2}
    \end{gather}
\end{lemma}

\begin{proof}
    For $\rho=0$ we have $y_t=x_t$ and all claims are trivial, so let $\rho>0$.
    The third admissibility condition gives
    $\kappa\norm{y_t-x_t}=\rho\kappa\norm{\nabla f(x_t)}/M(x_t)\le\rho$. The second
    condition, applied from $x_t$ to $y_t$ and from $y_t$ to $x_t$, gives
    $M(y_t)\le M(x_t)e^{\kappa\norm{y_t-x_t}}\le e^{\rho}M(x_t)$ and
    $M(x_t)\le e^{\rho}M(y_t)$. \Cref{lem:local} with $d=y_t-x_t$ gives
    \begin{align*}
        \norm{\nabla f(y_t)-\nabla f(x_t)}
        &\le M(x_t)\,\frac{e^{\kappa\norm{y_t-x_t}}-1}{\kappa\norm{y_t-x_t}}\,
        \frac{\rho\norm{\nabla f(x_t)}}{M(x_t)}\\
        &\le\rho\norm{\nabla f(x_t)}\,\frac{e^{\rho}-1}{\rho}
        =(e^{\rho}-1)\norm{\nabla f(x_t)},
    \end{align*}
    since $u\mapsto(e^u-1)/u$ is increasing and $\kappa\norm{y_t-x_t}\le\rho$.
    The claims in \eqref{eq:geometry2} follow from the triangle and
    Cauchy--Schwarz inequalities:
    \begin{align*}
        \norm{\nabla f(y_t)}
        &\le\norm{\nabla f(x_t)}+\norm{\nabla f(y_t)-\nabla f(x_t)}
        \le e^{\rho}\norm{\nabla f(x_t)},\\
        \ip{\nabla f(x_t)}{\nabla f(y_t)}
        &=\norm{\nabla f(x_t)}^2+\ip{\nabla f(x_t)}{\nabla f(y_t)-\nabla f(x_t)}\\
        &\ge\norm{\nabla f(x_t)}^2-(e^{\rho}-1)\norm{\nabla f(x_t)}^2 .\qedhere
    \end{align*}
\end{proof}

\begin{lemma}
    \label{lem:descent}
    Let $M$ be $\kappa$-admissible for $f$, let $\rho\ge0$ and $\gamma>0$ satisfy
    $\chi(\rho,\gamma)>0$, and let $x_t$ be an iterate of \ref{eq:cssam}. Then
    \begin{align}
        f(x_{t+1})\le f(x_t)-\chi(\rho,\gamma)\frac{\gamma}{M(y_t)}\norm{\nabla f(x_t)}^2
        \le f(x_t)-\gamma e^{-\rho}\chi(\rho,\gamma)\,\frac{\norm{\nabla f(x_t)}^2}{M(x_t)} .
        \label{eq:descent}
    \end{align}
    In particular, $f(x_t)-f_\star$ is nonincreasing in $t$.
\end{lemma}

\begin{proof}
    By the third admissibility condition at $y_t$, the displacement
    $x_{t+1}-x_t=-\gamma\nabla f(y_t)/M(y_t)$ satisfies
    \begin{align*}
        \kappa\norm{x_{t+1}-x_t}=\gamma\kappa\,\frac{\norm{\nabla f(y_t)}}{M(y_t)}\le\gamma .
    \end{align*}
    \Cref{lem:local} at $x_t$ and the monotonicity of $\omega$ give
    \begin{align*}
        f(x_{t+1})\le f(x_t)-\frac{\gamma}{M(y_t)}\ip{\nabla f(x_t)}{\nabla f(y_t)}
        +M(x_t)\,\omega(\gamma)\,\frac{\gamma^2}{M(y_t)^2}\norm{\nabla f(y_t)}^2 .
    \end{align*}
    By \Cref{lem:geometry},
    $\ip{\nabla f(x_t)}{\nabla f(y_t)}\ge(2-e^{\rho})\norm{\nabla f(x_t)}^2$,
    $M(x_t)/M(y_t)\le e^{\rho}$, and
    $\norm{\nabla f(y_t)}^2\le e^{2\rho}\norm{\nabla f(x_t)}^2$. Hence
    \begin{align*}
        f(x_{t+1})&\le f(x_t)-\frac{\gamma}{M(y_t)}
        \bigl[2-e^{\rho}-\gamma\omega(\gamma)e^{3\rho}\bigr]\norm{\nabla f(x_t)}^2\\
        &=f(x_t)-\chi(\rho,\gamma)\frac{\gamma}{M(y_t)}\norm{\nabla f(x_t)}^2 .
    \end{align*}
    Since $\chi(\rho,\gamma)>0$ and $M(y_t)\le e^{\rho}M(x_t)$, the second
    inequality in \eqref{eq:descent} follows.
\end{proof}

\begin{lemma}
    \label{lem:fejer}
    Let $f$ be convex with a minimizer $x_\star$, let $M$ be $\kappa$-admissible
    for $f$, let $\rho\ge0$ and $\gamma>0$, and let $x_t$ be an iterate
    of \ref{eq:cssam}. Then
    \begin{align}
        \norm{x_{t+1}-x_\star}^2\le\norm{x_t-x_\star}^2
        -2c_0\frac{\gamma}{M(y_t)}\bigl(f(x_t)-f_\star\bigr).
        \label{eq:fejer}
    \end{align}
    In particular, if $c_0\ge0$, then $\norm{x_t-x_\star}\le\norm{x_0-x_\star}$
    for all $t$.
\end{lemma}

\begin{proof}
    Expanding the square,
    \begin{align*}
        \norm{x_{t+1}-x_\star}^2=\norm{x_t-x_\star}^2
        -2\frac{\gamma}{M(y_t)}\ip{\nabla f(y_t)}{x_t-x_\star}
        +\frac{\gamma^2}{M(y_t)^2}\norm{\nabla f(y_t)}^2 .
    \end{align*}
    Convexity at $y_t$ gives $\ip{\nabla f(y_t)}{y_t-x_\star}\ge f(y_t)-f_\star$,
    and convexity at $x_t$ gives
    $f(y_t)\ge f(x_t)+\ip{\nabla f(x_t)}{y_t-x_t}=f(x_t)+\rho\norm{\nabla f(x_t)}^2/M(x_t)$.
    Since $x_t-x_\star=(y_t-x_\star)-\rho\nabla f(x_t)/M(x_t)$,
    \begin{align*}
        \ip{\nabla f(y_t)}{x_t-x_\star}
        &\ge f(y_t)-f_\star-\frac{\rho}{M(x_t)}\ip{\nabla f(y_t)}{\nabla f(x_t)}\\
        &\ge f(x_t)-f_\star-\frac{\rho}{M(x_t)}\ip{\nabla f(y_t)-\nabla f(x_t)}{\nabla f(x_t)}\\
        &\ge f(x_t)-f_\star-\frac{\rho(e^{\rho}-1)}{M(x_t)}\norm{\nabla f(x_t)}^2
        \ge\Bigl(1-\frac{\rho(e^{\rho}-1)}{2\log2-1}\Bigr)\bigl(f(x_t)-f_\star\bigr),
    \end{align*}
    where we used the Cauchy--Schwarz inequality with \Cref{lem:geometry},
    and then \Cref{lem:gap} in the form
    $\norm{\nabla f(x_t)}^2/M(x_t)\le(f(x_t)-f_\star)/(2\log2-1)$. Moreover, by
    \Cref{lem:geometry,lem:gap},
    \begin{align*}
        \frac{\norm{\nabla f(y_t)}^2}{M(y_t)}
        \le\frac{e^{2\rho}\norm{\nabla f(x_t)}^2}{e^{-\rho}M(x_t)}
        =e^{3\rho}\frac{\norm{\nabla f(x_t)}^2}{M(x_t)}
        \le\frac{e^{3\rho}}{2\log2-1}\bigl(f(x_t)-f_\star\bigr).
    \end{align*}
    Substituting both bounds into the expansion,
    \begin{align*}
        \norm{x_{t+1}-x_\star}^2\le\norm{x_t-x_\star}^2
        -2\frac{\gamma}{M(y_t)}\Bigl[1-\frac{\rho(e^{\rho}-1)}{2\log2-1}
        -\frac{\gamma e^{3\rho}}{2(2\log2-1)}\Bigr]\bigl(f(x_t)-f_\star\bigr),
    \end{align*}
    and the bracket equals $c_0$. The last claim follows since
    $f(x_t)-f_\star\ge0$.
\end{proof}

\medskip
\noindent
The next two lemmas concern the loss scale $M_H$ of \eqref{eq:assH} only. 
\begin{lemma}
    \label{lem:selfbound}
    Under \eqref{eq:assH}, for every $x\in\R^d$,
    \begin{align}
        \norm{\nabla f(x)}^2\le2H_0\bigl(f(x)-f_\star\bigr)+H_1\bigl(f(x)-f_\star\bigr)^2,
        \qquad
        \sqrt{H_1}\,\norm{\nabla f(x)}\le M_H(x).
        \label{eq:selfbound}
    \end{align}
\end{lemma}

\begin{proof}
    If $\nabla f(x)=0$ there is nothing to prove. Suppose first $H_1=0$. If
    $H_0=0$, then $\nabla^2f\equiv0$, so $f$ is affine and, being bounded below,
    constant, whence $\nabla f(x)=0$. If $H_0>0$, then $f$ is $H_0$-smooth and
    the descent lemma at the point $x-\nabla f(x)/H_0$ gives
    $f_\star\le f(x-\nabla f(x)/H_0)\le f(x)-\norm{\nabla f(x)}^2/(2H_0)$, which
    is the first claim; the second is trivial.

    \medskip
    \noindent
    Now let $H_1>0$. We compare the loss along the negative gradient direction with the solution of a scalar differential equation, using $f\ge f_\star$ to bound $\norm{\nabla f(x)}$. Define, for $t\ge0$, the function
    \begin{align*}
        \psi(t):=f\Bigl(x-t\,\frac{\nabla f(x)}{\norm{\nabla f(x)}}\Bigr)-f_\star+\frac{H_0}{H_1} .
    \end{align*}
    Then $\psi\ge H_0/H_1$, $\psi(0)=f(x)-f_\star+H_0/H_1$,
    $\psi'(0)=-\norm{\nabla f(x)}$, and, by \eqref{eq:assH},
    \begin{align*}
        \psi''(t)\le\Bigl\lVert\nabla^2f\Bigl(x-t\,\frac{\nabla f(x)}{\norm{\nabla f(x)}}\Bigr)\Bigr\rVert
    \le H_0+H_1\Bigl(f\Bigl(x-t\,\frac{\nabla f(x)}{\norm{\nabla f(x)}}\Bigr)-f_\star\Bigr)
    =H_1\psi(t).
    \end{align*}
    Let $q$ solve $q''=H_1q$ with $q(0)=\psi(0)$ and $q'(0)=\psi'(0)$, that is
    \begin{align*}
        q(t)=\psi(0)\cosh\bigl(\sqrt{H_1}\,t\bigr)
        -\frac{\norm{\nabla f(x)}}{\sqrt{H_1}}\sinh\bigl(\sqrt{H_1}\,t\bigr).
    \end{align*}
    The difference $\psi-q$ vanishes together with its derivative at $t=0$ and
    satisfies $(\psi-q)''-H_1(\psi-q)=\psi''-H_1\psi\le0$, so by the
    variation-of-constants formula
    \begin{align*}
        \psi(t)-q(t)=\int_0^t\frac{\sinh(\sqrt{H_1}(t-s))}{\sqrt{H_1}}
        \bigl(\psi''(s)-H_1\psi(s)\bigr)\,ds\le0 .
    \end{align*}
    Hence $q(t)\ge\psi(t)\ge H_0/H_1$ for all $t\ge0$. If
    $\norm{\nabla f(x)}>\sqrt{H_1}\,\psi(0)$, then $q(t)\to-\infty$ as
    $t\to\infty$, a contradiction; thus $\norm{\nabla f(x)}\le\sqrt{H_1}\,\psi(0)$,
    and then $\inf_{t\ge0}q(t)=\sqrt{\psi(0)^2-\norm{\nabla f(x)}^2/H_1}$. This
    infimum is at least $H_0/H_1$, so
    \begin{align*}
        \norm{\nabla f(x)}^2&\le H_1\Bigl(\psi(0)^2-\frac{H_0^2}{H_1^2}\Bigr)
        =H_1\Bigl(f(x)-f_\star+\frac{H_0}{H_1}\Bigr)^2-\frac{H_0^2}{H_1}\\
        &=H_1\bigl(f(x)-f_\star\bigr)^2+2H_0\bigl(f(x)-f_\star\bigr),
    \end{align*}
    which is the first claim. For the second, multiply the first claim by $H_1$:
    \begin{align*}
        H_1\norm{\nabla f(x)}^2\le2H_0H_1\bigl(f(x)-f_\star\bigr)+H_1^2\bigl(f(x)-f_\star\bigr)^2
        \le\bigl(H_0+H_1(f(x)-f_\star)\bigr)^2 .\qedhere
    \end{align*}
\end{proof}

\begin{lemma}
    \label{lem:zeroscale}
    Under \eqref{eq:assH}, if $M_H(x_0)=0$ for some $x_0\in\R^d$, then
    $M_H\equiv0$ and $f$ is constant.
\end{lemma}

\begin{proof}
    $M_H(x_0)=0$ forces $H_0=0$ and $f(x_0)=f_\star$. If $H_1=0$, then
    $\nabla^2f\equiv0$, so $f$ is affine and bounded below, hence constant, and
    $M_H\equiv0$. Let $H_1>0$ and fix $y\in\R^d$. The function
    $s\mapsto M_H(x_0+s(y-x_0))=H_1\bigl(f(x_0+s(y-x_0))-f_\star\bigr)$ is
    nonnegative and vanishes at $s=0$, and by \Cref{lem:selfbound} its
    derivative satisfies
    \begin{align*}
        \frac{d}{ds}M_H\bigl(x_0+s(y-x_0)\bigr)
        &=H_1\ip{\nabla f(x_0+s(y-x_0))}{y-x_0}\\
        &\le\sqrt{H_1}\,M_H\bigl(x_0+s(y-x_0)\bigr)\norm{y-x_0} .
    \end{align*}
    \Cref{lem:gronwall} with $c=\sqrt{H_1}\norm{y-x_0}$ gives
    $M_H(x_0+s(y-x_0))\le M_H(x_0)e^{\sqrt{H_1}s\norm{y-x_0}}=0$ on $[0,1]$, so
    $M_H(y)=0$, that is $f(y)=f_\star$. As $y$ was arbitrary, $f$ is constant.
\end{proof}

\medskip
\noindent
We close this section with the elementary sequence lemma that turns a quadratic decrease of the gap into a $1/T$ rate. 
\begin{lemma}
    \label{lem:recursion}
    Let $(x_t)_{t\ge0}$ be a sequence such that $f(x_t)-f_\star$ is nonincreasing
    in $t$ and
    \begin{align*}
        f(x_{t+1})-f_\star\le\bigl(f(x_t)-f_\star\bigr)-\beta\bigl(f(x_t)-f_\star\bigr)^2
        \qquad\text{for all }t\ge0,
    \end{align*}
    where $\beta>0$. Then $f(x_T)-f_\star\le1/(\beta T)$ for every $T\ge1$.
\end{lemma}

\begin{proof}
    If $f(x_T)=f_\star$ there is nothing to prove. Otherwise
    $f(x_t)-f_\star\ge f(x_T)-f_\star>0$ for all $t\le T$, and for $t<T$,
    \begin{align*}
        \frac1{f(x_{t+1})-f_\star}-\frac1{f(x_t)-f_\star}
        =\frac{\bigl(f(x_t)-f_\star\bigr)-\bigl(f(x_{t+1})-f_\star\bigr)}
        {\bigl(f(x_t)-f_\star\bigr)\bigl(f(x_{t+1})-f_\star\bigr)}
        \ge\beta\,\frac{f(x_t)-f_\star}{f(x_{t+1})-f_\star}\ge\beta .
    \end{align*}
    Summing over $t<T$ gives
    $1/(f(x_T)-f_\star)\ge1/(f(x_0)-f_\star)+\beta T\ge\beta T$.
\end{proof}

\newpage
\section{Proofs for \texorpdfstring{\Cref{sec:failures}}{Section 2}}
\label{app:endpoint}

In this section, we present the proofs of the results of \Cref{sec:failures}. We first verify the distortion bounds \eqref{eq:heuristic-bounds-smooth} and \eqref{eq:heuristic-bounds-l0l1} for \ref{eq:sam} and \ref{eq:usam}. We then prove \Cref{prop:endpoint-failures} by direct computation on the two one-dimensional examples, \Cref{thm:necessary-scaling} by exhibiting two one-dimensional objectives in the class \eqref{eq:assL}, and \Cref{prop:admissible} using the lemmas of \Cref{app:lemmas}. For completeness, we restate each result before its proof.

\paragraph{Verification of \eqref{eq:heuristic-bounds-smooth} and \eqref{eq:heuristic-bounds-l0l1}.}
The displacement is $d=r\nabla f(x)/\norm{\nabla f(x)}$
for \ref{eq:sam} and $d=\rho\nabla f(x)$ for \ref{eq:usam}, so $\norm d=r$ and
$\norm d=\rho\norm{\nabla f(x)}$ respectively. Under \eqref{eq:assSmooth},
$\norm{\nabla f(x+d)-\nabla f(x)}\le L\norm d$, and dividing by
$\norm{\nabla f(x)}$ gives \eqref{eq:heuristic-bounds-smooth}. Under
\eqref{eq:assL} with $L_1>0$, the scale $M_L$ is $L_1$-admissible by
\Cref{prop:admissible}(ii), which is proved below, so the first inequality in \Cref{lem:local} gives
\begin{align*}
    \norm{\nabla f(x+d)-\nabla f(x)}
    \le\bigl(L_0+L_1\norm{\nabla f(x)}\bigr)\,\frac{e^{L_1\norm d}-1}{L_1} .
\end{align*}
Inserting the two values of $\norm d$ and dividing by $\norm{\nabla f(x)}$
gives the two bounds in \eqref{eq:heuristic-bounds-l0l1}.

\begin{restatement}{prop:endpoint-failures}
    \begin{enumerate}[label=(\roman*)]
        \item
        Let $f(x)=x^2/2$, which satisfies \eqref{eq:assSmooth} with $L=1$. Run \ref{eq:sam} with $r>0$ and $\eta\in(0,2)$. For every $x\ne0$,
        \begin{equation*}
            D_{\mathrm{SAM}}(x):=D(x,r\sign(x))=\frac{r}{|x|},
            \qquad
            x_{t+1}=\bigl[1-\eta\,(1+D_{\mathrm{SAM}}(x_t))\bigr]\,x_t .
        \end{equation*}
        Let $a=\eta r/(2-\eta)$. Then $|x_{t+1}|<|x_t|$ if $|x_t|>a$, $|x_{t+1}|>|x_t|$ if $0<|x_t|<a$, and $x_{t+1}=-x_t$ if $|x_t|=a$. Consequently $\pm a$ is a two-cycle, and no trajectory converges to the minimizer (unless it lands on it exactly).

        \item
        Let $f(x)=x^4/4$, which satisfies \eqref{eq:assL} with $(L_0,L_1)=(4,1)$. Run \ref{eq:usam} with $\rho>0$ and $\eta>0$. For every $x\ne0$,
        \begin{equation*}
            D_{\mathrm{USAM}}(x):=D(x,\rho x^3)=(1+\rho x^2)^3-1,
            \qquad
            x_{t+1}=\bigl[1-\eta\,(1+D_{\mathrm{USAM}}(x_t))\,x_t^2\bigr]\,x_t .
        \end{equation*}
        Consequently, if $|x_0|>(2/(\eta\rho^3))^{1/8}$, then $x_{t+1}$ has the opposite sign of $x_t$ and $|x_{t+1}|>|x_t|$ for every $t$, and $|x_t|\to\infty$.
    \end{enumerate}
\end{restatement}

\begin{proof}[Proof of \Cref{prop:endpoint-failures}]
    \emph{(i)} Here $\nabla f(x)=x$, so for $x\ne0$ the perturbed point of
    \ref{eq:sam} is $y=x+r\sign(x)$, and since $r\sign(x)=(r/|x|)\,x$,
    \begin{align*}
        D_{\mathrm{SAM}}(x)=\frac{|y-x|}{|x|}=\frac{r}{|x|},
        \qquad
        x^+=x-\eta y=x-\eta\Bigl(1+\frac{r}{|x|}\Bigr)x
        =\bigl[1-\eta\,(1+D_{\mathrm{SAM}}(x))\bigr]\,x .
    \end{align*}
    Write $c(x):=1-\eta-\eta r/|x|$ for the factor in brackets, so that
    $|x^+|=|c(x)|\,|x|$. Since $\eta>0$ we have $c(x)<1$, and since $\eta<2$,
    \begin{align*}
        c(x)>-1
        \iff 2-\eta>\frac{\eta r}{|x|}
        \iff |x|>\frac{\eta r}{2-\eta}=a .
    \end{align*}
    Hence $|c(x)|<1$ if $|x|>a$, $|c(x)|>1$ if $0<|x|<a$, and $c(x)=-1$ if
    $|x|=a$, which are the three claims. In particular $a\mapsto-a\mapsto a$,
    so $\pm a$ is a two-cycle. Finally, suppose $x_t\ne0$ for all $t$ and
    $x_t\to0$. Then $0<|x_t|<a$ for all large $t$, so $|x_t|$ is eventually
    strictly increasing, which contradicts $x_t\to0$.

    \emph{(ii)} Here $f'(x)=x^3$ and $f''(x)=3x^2$. For $u\ge0$ the function
    $3u^2-u^3$ attains its maximum $4$ at $u=2$, so
    $3x^2\le4+|x|^3=4+|f'(x)|$ and \eqref{eq:assL} holds with $(L_0,L_1)=(4,1)$.
    The perturbed point of \ref{eq:usam} is $y=x+\rho x^3=x(1+\rho x^2)$, so
    that $y^3=(1+\rho x^2)^3x^3$ and, for $x\ne0$,
    \begin{align*}
        D_{\mathrm{USAM}}(x)=\frac{|y^3-x^3|}{|x|^3}=(1+\rho x^2)^3-1,
        \qquad
        x^+=x-\eta y^3=\bigl[1-\eta\,(1+D_{\mathrm{USAM}}(x))\,x^2\bigr]\,x .
    \end{align*}
    If $|x|>(2/(\eta\rho^3))^{1/8}$, then
    $\eta(1+\rho x^2)^3x^2\ge\eta\rho^3x^8>2$, so the factor in brackets is
    less than $-1$: $x^+$ has the opposite sign of $x$ and $|x^+|>|x|$. In
    particular $|x^+|>(2/(\eta\rho^3))^{1/8}$ again, so by induction the sign
    alternates and $|x_t|$ increases strictly for every $t$. Moreover,
    $|x_{t+1}|\ge(\eta\rho^3x_t^8-1)|x_t|\ge(\eta\rho^3x_0^8-1)|x_t|$ with
    $\eta\rho^3x_0^8-1>1$, so $|x_t|\to\infty$.
\end{proof}

\begin{restatement}{thm:necessary-scaling}
    Let $L_0>0$, $L_1>0$, and consider the perturbation given by a nonnegative collinear rule $d(x)=\alpha(\norm{\nabla f(x)})\nabla f(x)$. Suppose that for every one-dimensional objective satisfying \eqref{eq:assL} and every point with nonzero gradient, $D(x,d(x))\le\delta$. Then, for every $s>0$,
    \begin{equation}\tag{\ref{eq:necessary}}
        \alpha(s)\le \min\left\{\frac{\delta}{L_0},
        \frac{\log(1+\delta)}{L_1s}\right\}.
        \end{equation}
\end{restatement}

\begin{proof}[Proof of \Cref{thm:necessary-scaling}]
    Fix $s>0$. Consider first $f_q(x)=L_0x^2/2$. It is bounded below with
    $f_q''\equiv L_0\le L_0+L_1|f_q'(x)|$, so it satisfies \eqref{eq:assL}. At
    the point $x=s/L_0$ we have $f_q'(x)=s$, the perturbation is
    $d(x)=\alpha(s)s$, and
    \begin{align*}
        D\bigl(x,d(x)\bigr)=\frac{|f_q'(x+d(x))-f_q'(x)|}{|f_q'(x)|}
        =\frac{L_0\alpha(s)s}{s}=L_0\alpha(s) .
    \end{align*}
    The hypothesis $D(x,d(x))\le\delta$ gives $\alpha(s)\le\delta/L_0$.

    \medskip
    \noindent
    Next consider $f_e(x)=e^{L_1x}/L_1^2$. It is positive, hence bounded
    below, with $f_e'(x)=e^{L_1x}/L_1>0$ and
    $f_e''(x)=e^{L_1x}=L_1f_e'(x)\le L_0+L_1|f_e'(x)|$, so it satisfies
    \eqref{eq:assL}. At the point $x=\log(L_1s)/L_1$ we have $f_e'(x)=s$, the
    perturbation is $d(x)=\alpha(s)s\ge0$, and since
    $f_e'(x+d)=e^{L_1d}f_e'(x)$,
    \begin{align*}
        D\bigl(x,d(x)\bigr)=\frac{f_e'(x+d(x))-f_e'(x)}{f_e'(x)}
        =e^{L_1\alpha(s)s}-1 .
    \end{align*}
    The hypothesis gives $e^{L_1\alpha(s)s}\le1+\delta$, that is
    $\alpha(s)\le\log(1+\delta)/(L_1s)$. Together with the first bound this
    is \eqref{eq:necessary}.
\end{proof}

\begin{restatement}{prop:admissible}
    \leavevmode
    \begin{enumerate}[label=(\roman*)]
        \item Under \eqref{eq:assSmooth}, the constant scale $M\equiv L$ is
        admissible with $\kappa=0$.

        \item Under \eqref{eq:assL}, the scale $M_L$ is admissible with
        $\kappa=L_1$.

        \item Under \eqref{eq:assH}, either $M_H\equiv0$ and $f$ is constant,
        or $M_H>0$ everywhere and the scale $M_H$ is admissible with
        $\kappa=\sqrt{H_1}$.

        \item Under \eqref{eq:assL}, for every $x\in\R^d$,
        \begin{align}\tag{\ref{eq:L-implies-H}}
            L_0+L_1\norm{\nabla f(x)}
            \le\frac32L_0+\frac{3L_1^2}{2\,(2\log2-1)}\bigl(f(x)-f_\star\bigr).
        \end{align}
        In particular, \eqref{eq:assL} implies \eqref{eq:assH} with the
        constants $H_0=3L_0/2$ and $H_1=3L_1^2/(2(2\log2-1))$.
    \end{enumerate}
\end{restatement}

\begin{proof}[Proof of \Cref{prop:admissible}]
    \emph{(i)} With $M\equiv L$ and $\kappa=0$, the three conditions in
    \eqref{eq:admissible} read $\norm{\nabla^2f(x)}\le L$, $L\le L$, and $0\le L$.

    \emph{(ii)} The first condition is \eqref{eq:assL}, and the third holds
    with $\kappa=L_1$ because $L_1\norm{\nabla f(x)}\le L_0+L_1\norm{\nabla f(x)}$.
    For the second, fix $x,y$. Since $\nabla f$ is continuously differentiable,
    the function $s\mapsto M_L(x+s(y-x))=L_0+L_1\norm{\nabla f(x+s(y-x))}$ is
    Lipschitz on $[0,1]$, hence absolutely continuous, and at every point of
    differentiability
    \begin{align*}
        \frac{d}{ds}M_L\bigl(x+s(y-x)\bigr)
        &\le L_1\norm{\nabla^2f(x+s(y-x))\,(y-x)}
        \le L_1\norm{\nabla^2f(x+s(y-x))}\norm{y-x}\\
        &\le L_1M_L\bigl(x+s(y-x)\bigr)\norm{y-x} .
    \end{align*}
    \Cref{lem:gronwall} with $c=L_1\norm{y-x}$ gives $M_L(y)\le M_L(x)e^{L_1\norm{y-x}}$.

    \emph{(iii)} If $M_H$ vanishes somewhere, \Cref{lem:zeroscale} gives
    $M_H\equiv0$ and $f$ constant. Otherwise $M_H>0$ everywhere. The first
    condition is \eqref{eq:assH}, and the third, with $\kappa=\sqrt{H_1}$,
    is \Cref{lem:selfbound}. For the second, fix $x,y$; by
    \Cref{lem:selfbound},
    \begin{align*}
        \frac{d}{ds}M_H\bigl(x+s(y-x)\bigr)
        &=H_1\ip{\nabla f(x+s(y-x))}{y-x}
        \le H_1\norm{\nabla f(x+s(y-x))}\norm{y-x}\\
        &\le\sqrt{H_1}\,M_H\bigl(x+s(y-x)\bigr)\norm{y-x},
    \end{align*}
    and \Cref{lem:gronwall} with $c=\sqrt{H_1}\norm{y-x}$ gives $M_H(y)\le M_H(x)e^{\sqrt{H_1}\norm{y-x}}$.

    \emph{(iv)} By (ii), \Cref{lem:gap} applies to $M_L$ and gives
    $(2\log2-1)\norm{\nabla f(x)}^2\le(f(x)-f_\star)\bigl(L_0+L_1\norm{\nabla f(x)}\bigr)$,
    that is
    \begin{align*}
        \norm{\nabla f(x)}^2-\frac{L_1(f(x)-f_\star)}{2\log2-1}\norm{\nabla f(x)}
        -\frac{L_0(f(x)-f_\star)}{2\log2-1}\le0 .
    \end{align*}
    Since $s^2\le A+Bs$ with $A,B,s\ge0$ implies
    $s\le\bigl(B+\sqrt{B^2+4A}\bigr)/2\le B+\sqrt A$,
    \begin{align*}
        \norm{\nabla f(x)}\le\frac{L_1(f(x)-f_\star)}{2\log2-1}
        +\sqrt{\frac{L_0(f(x)-f_\star)}{2\log2-1}} .
    \end{align*}
    Therefore, using $\sqrt{AB}\le(A+B)/2$ with $A=L_0$ and
    $B=L_1^2(f(x)-f_\star)/(2\log2-1)$,
    \begin{align*}
        L_0+L_1\norm{\nabla f(x)}
        &\le L_0+\frac{L_1^2(f(x)-f_\star)}{2\log2-1}
        +\sqrt{L_0\cdot\frac{L_1^2(f(x)-f_\star)}{2\log2-1}}\\
        &\le\frac32L_0+\frac{3L_1^2}{2(2\log2-1)}\bigl(f(x)-f_\star\bigr),
    \end{align*}
    which is \eqref{eq:L-implies-H}. Combined with \eqref{eq:assL}, this
    bounds $\norm{\nabla^2f(x)}$ by the same right-hand side, which is
    \eqref{eq:assH} with the stated constants.
\end{proof}

\begin{restatement}{thm:relative-control}
    Let $M$ be $\kappa$-admissible, let $\rho\ge0$, $\gamma>0$, and run \ref{eq:cssam}. Then every iteration with $\nabla f(x_t)\ne0$ satisfies
    \begin{equation*}
        e^{-\rho}M(x_t)\le M(y_t)\le e^\rho M(x_t),
        \qquad
        \frac{\norm{\nabla f(y_t)-\nabla f(x_t)}}{\norm{\nabla f(x_t)}}\le e^\rho-1.
    \end{equation*}
    In particular, $D_{\mathrm{CS\text{-}SAM}}\le e^\rho-1$.
\end{restatement}

\begin{proof}[Proof of \Cref{thm:relative-control}]
    The two claims are the inequalities in \eqref{eq:geometry1} of \Cref{lem:geometry}, which holds for every $\kappa$-admissible scale $M$ and every $\rho\ge0$; since $\nabla f(x_t)\ne0$, dividing the second inequality by $\norm{\nabla f(x_t)}$ gives the distortion bound.
\end{proof}

\newpage
\section{Proofs for \texorpdfstring{\Cref{sec:rates}}{Section 3}}
\label{app:rates}

In this section, we present the proofs of \Cref{thm:L-rates} in \Cref{app:L} and of \Cref{thm:H-rates} in \Cref{app:H}; for completeness, each theorem is restated before its proofs. 

\subsection{\texorpdfstring{Proofs for \Cref{sec:L}: the $(L_0,L_1)$ curvature scale}{Proofs for Section 3.1: the (L0,L1) curvature scale}}
\label{app:L}

Throughout this subsection $f$ satisfies \eqref{eq:assL} and $M=M_L$,
which is admissible with $\kappa=L_1$ by \Cref{prop:admissible}(ii),
so \Crefrange{lem:local}{lem:fejer} apply with
$M(x_t)=L_0+L_1\norm{\nabla f(x_t)}$ and $M(y_t)=L_0+L_1\norm{\nabla f(y_t)}$.
We use repeatedly that the map $s\mapsto s^2/(L_0+L_1s)$ is increasing on
$[0,\infty)$, since its derivative is $(2L_0s+L_1s^2)/(L_0+L_1s)^2\ge0$.
In this setting \Cref{lem:descent} reads
\begin{align}
    f(x_{t+1})-f_\star\le\bigl(f(x_t)-f_\star\bigr)
    -\gamma e^{-\rho}\chi(\rho,\gamma)\,
    \frac{\norm{\nabla f(x_t)}^2}{L_0+L_1\norm{\nabla f(x_t)}} .
    \label{eq:descent-L}
\end{align}

\begin{restatement}{thm:L-rates}
    Let $f$ satisfy \eqref{eq:assL}, let $\rho\ge0$ and $\gamma>0$ satisfy $\chi(\rho,\gamma)>0$, and let $(x_t)_{t\ge0}$ be generated by \ref{eq:cssam} with $M=M_L$. Then, for every $T\ge1$, the following hold:
    \begin{enumerate}[label=(\roman*)]
        \item
        \textit{(Nonconvex rate)}
        \begin{align}\tag{\ref{eq:L-nonconvex}}
            \min_{0\le t<T}\norm{\nabla f(x_t)}^2
            \le\left(\sqrt{\frac{e^{\rho}L_0\,\bigl(f(x_0)-f_\star\bigr)}{\gamma\chi(\rho,\gamma)\,T}}
            +\frac{e^{\rho}L_1\,\bigl(f(x_0)-f_\star\bigr)}{\gamma\chi(\rho,\gamma)\,T}\right)^{2}.
        \end{align}
        \item
        \textit{(Convex rate)} If in addition $f$ is convex with a minimizer $x_\star$ and $c_0\ge0$, then
        \begin{align}\tag{\ref{eq:L-convex}}
            f(x_T)-f_\star
            \le\frac{e^{\rho}\Bigl(L_0\norm{x_0-x_\star}^2
            +L_1\norm{x_0-x_\star}\bigl(f(x_0)-f_\star\bigr)\Bigr)}
            {\gamma\chi(\rho,\gamma)\,T}.
        \end{align}
        \item
        \textit{(PL rate)} If in addition $f$ satisfies the $\mu$-PL condition \eqref{eq:PL}, then
        \begin{align}\tag{\ref{eq:L-PL-a}}
            f(x_T)-f_\star
            \le\left(1-\frac{2\mu\gamma e^{-\rho}\chi(\rho,\gamma)}
            {L_0+L_1\sqrt{2\mu\,\bigl(f(x_0)-f_\star\bigr)}}\right)^{T}
            \bigl(f(x_0)-f_\star\bigr).
        \end{align}
    \end{enumerate}
\end{restatement}

\begin{proof}[Proof of \Cref{thm:L-rates}]
    \begin{enumerate}[label=(\roman*)]
        \item Summing \eqref{eq:descent-L} over $t<T$ gives
        \begin{align*}
            \gamma e^{-\rho}\chi(\rho,\gamma)
            \sum_{t<T}\frac{\norm{\nabla f(x_t)}^2}{L_0+L_1\norm{\nabla f(x_t)}}
            \le f(x_0)-f(x_T)\le f(x_0)-f_\star .
        \end{align*}
        Since the map $s\mapsto s^2/(L_0+L_1s)$ is increasing, each summand is
        bounded below by the value of this map at the smallest of the gradient
        norms, hence
        \begin{align*}
            \Bigl(\min_{t<T}\norm{\nabla f(x_t)}\Bigr)^2
            \le\frac{e^{\rho}\bigl(f(x_0)-f_\star\bigr)}{\gamma\chi(\rho,\gamma)\,T}
            \Bigl(L_0+L_1\min_{t<T}\norm{\nabla f(x_t)}\Bigr).
        \end{align*}
        Since $s^2\le A+Bs$ with $A,B,s\ge0$ implies $s\le B+\sqrt A$, we get
        \begin{align*}
            \min_{t<T}\norm{\nabla f(x_t)}
            \le\sqrt{\frac{e^{\rho}L_0\bigl(f(x_0)-f_\star\bigr)}{\gamma\chi(\rho,\gamma)\,T}}
            +\frac{e^{\rho}L_1\bigl(f(x_0)-f_\star\bigr)}{\gamma\chi(\rho,\gamma)\,T},
        \end{align*}
        and squaring gives \eqref{eq:L-nonconvex}.

        \item Since $c_0\ge0$, \Cref{lem:fejer} gives
        $\norm{x_t-x_\star}\le\norm{x_0-x_\star}$ for all $t$. Convexity and the
        Cauchy--Schwarz inequality give
        \begin{align*}
            f(x_t)-f_\star\le\ip{\nabla f(x_t)}{x_t-x_\star}\le\norm{\nabla f(x_t)}\norm{x_0-x_\star} .
        \end{align*}
        Since the map $s\mapsto s^2/(L_0+L_1s)$ is increasing and
        $f(x_t)-f_\star\le f(x_0)-f_\star$,
        \begin{align*}
            \frac{\norm{\nabla f(x_t)}^2}{L_0+L_1\norm{\nabla f(x_t)}}
            &\ge\frac{\bigl(f(x_t)-f_\star\bigr)^2}
            {L_0\norm{x_0-x_\star}^2+L_1\norm{x_0-x_\star}\bigl(f(x_t)-f_\star\bigr)}\\
            &\ge\frac{\bigl(f(x_t)-f_\star\bigr)^2}
            {L_0\norm{x_0-x_\star}^2+L_1\norm{x_0-x_\star}\bigl(f(x_0)-f_\star\bigr)} .
        \end{align*}
        Inserting this into \eqref{eq:descent-L} gives
        \begin{align*}
            f(x_{t+1})-f_\star&\le\bigl(f(x_t)-f_\star\bigr)\\
            &\quad-\frac{\gamma e^{-\rho}\chi(\rho,\gamma)}
            {L_0\norm{x_0-x_\star}^2+L_1\norm{x_0-x_\star}\bigl(f(x_0)-f_\star\bigr)}
            \bigl(f(x_t)-f_\star\bigr)^2 .
        \end{align*}
        Since $f(x_t)-f_\star$ is nonincreasing by \Cref{lem:descent},
        \Cref{lem:recursion} applies and gives \eqref{eq:L-convex}.

        \item If $f(x_0)=f_\star$ all iterates are minimizers and there is nothing to prove,
        so let $f(x_0)>f_\star$. The PL condition \eqref{eq:PL} gives
        $\norm{\nabla f(x_t)}\ge\sqrt{2\mu(f(x_t)-f_\star)}$, so by the monotonicity
        of $s\mapsto s^2/(L_0+L_1s)$, \eqref{eq:descent-L}, and
        $f(x_t)-f_\star\le f(x_0)-f_\star$,
        \begin{align*}
            f(x_{t+1})-f_\star
            &\le\bigl(f(x_t)-f_\star\bigr)-\gamma e^{-\rho}\chi(\rho,\gamma)\,
            \frac{2\mu\bigl(f(x_t)-f_\star\bigr)}{L_0+L_1\sqrt{2\mu(f(x_t)-f_\star)}}\\
            &\le\Bigl(1-\frac{2\mu\gamma e^{-\rho}\chi(\rho,\gamma)}
            {L_0+L_1\sqrt{2\mu(f(x_0)-f_\star)}}\Bigr)\bigl(f(x_t)-f_\star\bigr).
        \end{align*}
        Applying this with $t=0$ and using $f(x_1)-f_\star\ge0$ and
        $f(x_0)-f_\star>0$ shows that the factor in parentheses is nonnegative, and
        it is less than one since $\mu\gamma e^{-\rho}\chi(\rho,\gamma)>0$. Induction
        on $t$ then gives \eqref{eq:L-PL-a}.
    \end{enumerate}
\end{proof}

\subsection{\texorpdfstring{Proofs for \Cref{sec:H}: the $(H_0,H_1)$ curvature scale}{Proofs for Section 3.2: the (H0,H1) curvature scale}}
\label{app:H}

Throughout this subsection $f$ satisfies \eqref{eq:assH}, is nonconstant,
and $M=M_H$, which is positive and admissible with $\kappa=\sqrt{H_1}$ by
\Cref{prop:admissible}(iii), so \Crefrange{lem:local}{lem:fejer} apply with $M(x_t)=H_0+H_1(f(x_t)-f_\star)$ and
$M(y_t)=H_0+H_1(f(y_t)-f_\star)$. Since $f(x_t)-f_\star$ is nonincreasing by
\Cref{lem:descent}, we have $M(x_t)\le H_0+H_1(f(x_0)-f_\star)$ for all
$t$, and \Cref{lem:descent} gives
\begin{align}
    \begin{gathered}
        f(x_{t+1})-f_\star\le\bigl(f(x_t)-f_\star\bigr)
        -\gamma e^{-\rho}\chi(\rho,\gamma)\,
        \frac{\norm{\nabla f(x_t)}^2}{H_0+H_1(f(x_t)-f_\star)}\\
        \le\bigl(f(x_t)-f_\star\bigr)
        -\gamma e^{-\rho}\chi(\rho,\gamma)\,
        \frac{\norm{\nabla f(x_t)}^2}{H_0+H_1(f(x_0)-f_\star)} .
    \end{gathered}
    \label{eq:descent-H}
\end{align}
Note that the second inequality in \eqref{eq:descent-H} has a constant coefficient in front of $\norm{\nabla f(x_t)}^2$, which is why the $H_1$ contribution enters the rates of \Cref{thm:H-rates} at the leading order, in contrast to \Cref{thm:L-rates}.

\begin{restatement}{thm:H-rates}
    Let $f$ satisfy \eqref{eq:assH} and be nonconstant, let $\rho\ge0$ and $\gamma>0$ satisfy $\chi(\rho,\gamma)>0$, and let $(x_t)_{t\ge0}$ be generated by \ref{eq:cssam} with $M=M_H$. Then, for every $T\ge1$, the following hold:
    \begin{enumerate}[label=(\roman*)]
        \item
        \textit{(Nonconvex rate)}
        \begin{align}\tag{\ref{eq:H-nonconvex}}
            \min_{0\le t<T}\norm{\nabla f(x_t)}^2
            \le\frac{e^{\rho}\bigl(H_0+H_1(f(x_0)-f_\star)\bigr)\bigl(f(x_0)-f_\star\bigr)}
            {\gamma\chi(\rho,\gamma)\,T}.
        \end{align}
        \item
        \textit{(Convex rate)} If in addition $f$ is convex with a minimizer $x_\star$ and $c_0\ge0$, then
        \begin{align}\tag{\ref{eq:H-convex}}
            f(x_T)-f_\star
            \le\frac{e^{\rho}\bigl(H_0+H_1(f(x_0)-f_\star)\bigr)\norm{x_0-x_\star}^2}
            {\gamma\chi(\rho,\gamma)\,T}.
        \end{align}
        \item
        \textit{(PL rate)} If in addition $f$ satisfies the $\mu$-PL condition \eqref{eq:PL}, then
        \begin{align}\tag{\ref{eq:H-PL-a}}
            f(x_T)-f_\star
            \le\left(1-\frac{2\mu\gamma e^{-\rho}\chi(\rho,\gamma)}
            {H_0+H_1\bigl(f(x_0)-f_\star\bigr)}\right)^{T}
            \bigl(f(x_0)-f_\star\bigr).
        \end{align}
    \end{enumerate}
\end{restatement}

\begin{proof}[Proof of \Cref{thm:H-rates}]
    \begin{enumerate}[label=(\roman*)]
        \item Summing the second inequality in \eqref{eq:descent-H} over $t<T$ gives
        \begin{align*}
            \frac{\gamma e^{-\rho}\chi(\rho,\gamma)}{H_0+H_1(f(x_0)-f_\star)}
            \sum_{t<T}\norm{\nabla f(x_t)}^2
            \le f(x_0)-f(x_T)\le f(x_0)-f_\star ,
        \end{align*}
        and the minimum over $t<T$ is at most the average, which gives
        \eqref{eq:H-nonconvex}.
        
        \item Since $c_0\ge0$, \Cref{lem:fejer} gives
        $\norm{x_t-x_\star}\le\norm{x_0-x_\star}$, and convexity gives
        $\norm{\nabla f(x_t)}\ge(f(x_t)-f_\star)/\norm{x_0-x_\star}$ as in the proof
        of \Cref{thm:L-rates}(ii). Inserting this into the second inequality in
        \eqref{eq:descent-H},
        \begin{align*}
            f(x_{t+1})-f_\star\le\bigl(f(x_t)-f_\star\bigr)
            -\frac{\gamma e^{-\rho}\chi(\rho,\gamma)}
            {\norm{x_0-x_\star}^2\bigl(H_0+H_1(f(x_0)-f_\star)\bigr)}
            \bigl(f(x_t)-f_\star\bigr)^2,
        \end{align*}
        and \Cref{lem:recursion} gives \eqref{eq:H-convex}.

        \item Let $f(x_0)>f_\star$, the other case being trivial. The PL condition \eqref{eq:PL} and the
        second inequality in \eqref{eq:descent-H} give
        \begin{align*}
            f(x_{t+1})-f_\star
            \le\Bigl(1-\frac{2\mu\gamma e^{-\rho}\chi(\rho,\gamma)}{H_0+H_1(f(x_0)-f_\star)}\Bigr)
            \bigl(f(x_t)-f_\star\bigr).
        \end{align*}
        As in the proof of \Cref{thm:L-rates}(iii), the case $t=0$ shows that the
        factor in parentheses lies in $[0,1)$, and induction gives \eqref{eq:H-PL-a}.
    \end{enumerate}
\end{proof}

\newpage
\section{Experimental details}
\label{app:exps}

In this section, we give the details needed to reproduce the experiments of \Cref{sec:experiments}. \Cref{app:verification-details} specifies the one-dimensional verification experiments of \Cref{sec:verification}: the update rules, the curvature scales and the constants that make them valid everywhere, the stopping rules, and the evaluation of the distortion. \Cref{app:dnn-details} gives, for the deep-network experiments of \Cref{sec:dnn}, the training protocol, the minibatch curvature scales and their fixed coefficients, the tuning campaign with the complete grid of every method, the selection rule, and the selected configuration of every method in every model--dataset setting (\Cref{tab:selected}).

\medskip
\noindent
The pseudocode for \ref{eq:cssam} is given in \Cref{alg:cssam}.\begin{algorithm}[h]
    \caption{Curvature-Scaled SAM (\ref{eq:cssam}).}
    \label{alg:cssam}
    \begin{algorithmic}[1]
        \Require starting point $x_0\in\R^d$; perturbation multiplier $\rho\ge0$; step multiplier
        $\gamma>0$; admissible curvature scale $M$ (\Cref{def:admissible})
        \For{$t=0,1,\ldots$}
            \If{$\nabla f(x_t)=0$}
                \State \Return $x_t$ \Comment{stationary point reached; the method stops}
            \EndIf
            \State $y_t \gets x_t+\dfrac{\rho}{M(x_t)}\,\nabla f(x_t)$
            \State $x_{t+1} \gets x_t-\dfrac{\gamma}{M(y_t)}\,\nabla f(y_t)$
        \EndFor
    \end{algorithmic}
\end{algorithm}

\subsection{Verification experiments of \texorpdfstring{\Cref{sec:verification}}{Section 4.1}}
\label{app:verification-details}

The experiments of \Cref{sec:verification} are deterministic and one-dimensional, with analytic function values and gradients, so there are no datasets, minibatches, random seeds, or hyperparameter searches. All computations use double precision.

\paragraph{Update rules and scales.}
The three methods are run exactly as stated in \eqref{eq:sam}, \eqref{eq:usam}, and \eqref{eq:cssam}, with $\nabla f=f'$ and the norm equal to the absolute value. Both objectives attain $f_\star=0$ at $x=0$, and the two scales are
\begin{equation*}
    M_L(x)=L_0+L_1|f'(x)|,
    \qquad
    M_H(x)=H_0+H_1\bigl(f(x)-f_\star\bigr).
\end{equation*}
For the quadratic $f(x)=x^2/2$ we have $f''\equiv1$ and take $(L_0,L_1)=(H_0,H_1)=(1,0)$, so that $M_L=M_H\equiv1$ and a single \ref{eq:cssam} curve is shown. For the quartic $f(x)=x^4/4$ we have $f''(x)=3x^2$ and take $(L_0,L_1)=(4,1)$ and $(H_0,H_1)=(3,3)$, that is, $M_L(x)=4+|x|^3$ and $M_H(x)=3+3x^4/4$. 

\paragraph{Parameter regime.}
The common choice $\rho=\gamma=1/8$ for \ref{eq:cssam} lies in the regime of \Cref{thm:L-rates,thm:H-rates}. With the constants of \eqref{eq:constants}, $2-e^{1/8}\approx0.867$ and $\gamma\omega(\gamma)e^{3\rho}=\tfrac18\,\omega(\tfrac18)\,e^{3/8}\approx0.095$, so that $\chi(1/8,1/8)\approx0.772>0$; moreover $\rho(e^{\rho}-1)/(2\log2-1)\approx0.043$ and $\gamma e^{3\rho}/(2(2\log2-1))\approx0.235$, so that $c_0\approx0.722\ge0$. Since both objectives are convex with the minimizer $x_\star=0$, the convergence of \ref{eq:cssam} in \Cref{fig:endpoint-quadratic,fig:endpoint-quartic} is the one guaranteed by \Cref{thm:L-rates}(ii) with $M=M_L$ and by \Cref{thm:H-rates}(ii) with $M=M_H$, and the bound $e^{\rho}-1$ in \Cref{fig:distortion-quadratic,fig:distortion-quartic} is that of \Cref{thm:relative-control}, which holds for every $\rho\ge0$.

\paragraph{Two-cycle and divergence (\Cref{fig:endpoint-quadratic,fig:endpoint-quartic}).}
A run stops when $|f'(x)|\le10^{-12}$, when an update no longer changes the double-precision iterate or when an iterate or a perturbed point has absolute value at least $10^{12}$. On the quadratic, \ref{eq:sam} uses $r=1$ and $\eta=0.2$, and \ref{eq:cssam} uses $\rho=\gamma=1/8$, both from $x_0=3$ with at most 500 updates. The \ref{eq:sam} recurrence is $x_{t+1}=0.8x_t-0.2\sign(x_t)$, whose two-cycle has amplitude $a=\eta r/(2-\eta)=1/9$, marked by the dashed lines, and the \ref{eq:cssam} recurrence is $x_{t+1}=(1-\gamma(1+\rho))x_t=\tfrac{55}{64}x_t$. \ref{eq:sam} completes the 500 updates, whereas \ref{eq:cssam} reaches the gradient tolerance after 190 updates. On the quartic, \ref{eq:usam} uses $\rho=1/8$ and $\eta=1/32$ with at most 20 updates, and \ref{eq:cssam} uses $\rho=\gamma=1/8$ with either scale and at most $10{,}000$ updates, all from $x_0=4$. For \ref{eq:usam}, the first perturbed point is $y_0=4+\tfrac18\cdot4^3=12$, so that $x_1=4-\tfrac1{32}\cdot12^3=-50$, then $x_2\approx1.2\times10^{11}$, and the next perturbed point exceeds the cutoff; the trace therefore ends after two completed updates, and the plotted curve consists of its three recorded points. Both \ref{eq:cssam} traces complete the $10{,}000$ updates. \Cref{fig:endpoint-quartic} shows $|x_t|$ against the number $2t$ of gradient evaluations, two per completed update, over the first 50 updates. Note that the cutoff is a numerical termination rule; the divergence itself is the statement of \Cref{prop:endpoint-failures}(ii), whose threshold $(2/(\eta\rho^3))^{1/8}\approx3.67$ is exceeded by $x_0=4$.

\paragraph{Distortion (\Cref{fig:distortion-quadratic,fig:distortion-quartic}).}
These figures evaluate independent points rather than trajectories. At each nonstationary $x$ we form the perturbed point $y$ of the method and record the gradient norm $s=|f'(x)|$ and the distortion $D=|f'(y)-f'(x)|/|f'(x)|$ of \eqref{eq:distortion}; no outer step enters. The perturbations are those of \ref{eq:sam} with $r=1/8$ on the quadratic, of \ref{eq:usam} with $\rho=1/8$ on the quartic, and of \ref{eq:cssam} with $\rho=1/8$ on both, with each scale separately on the quartic. Note that the \ref{eq:sam} radius is $1$ in the two-cycle experiment and $1/8$ here; the two are separate experiments, and a common numerical value of the parameters does not mean equal perturbation lengths across methods. The dashed line is the bound $e^{\rho}-1=e^{1/8}-1\approx0.133$ of \Cref{thm:relative-control}. 

\subsection{Deep-network experiments of \texorpdfstring{\Cref{sec:dnn}}{Section 4.2}}
\label{app:dnn-details}

\paragraph{Training protocol.}
All runs train ResNet-20 or ResNet-32 with batch normalization on CIFAR-10 or CIFAR-100 for 200 epochs with batch size 128 and the cross-entropy loss, with random-crop and horizontal-flip augmentation and no cutout. Momentum and weight decay are zero for every method. The test accuracy is evaluated after every epoch. Every method computes two minibatch gradients per iteration, one at the iterate and one at the perturbed point, so that equal numbers of epochs correspond to equal numbers of gradient evaluations. The runs were performed on four RTX 6000 Ada GPUs.

\paragraph{Learning-rate schedules.}
The schedule of \ref{eq:sam} and \ref{eq:usam} acts on the outer step $\eta_t$ only; the radius $r$ and the multiplier $\rho$ are constant throughout training. With the constant schedule, $\eta_t\equiv\eta$. With the cosine schedule, the outer step is held fixed within each epoch and annealed from $\eta$ towards $0$ over the $E=200$ epochs: during epoch $e\in\{0,\dots,E-1\}$,
\begin{equation}
    \eta_t=\frac{\eta}{2}\Bigl(1+\cos\frac{\pi e}{E}\Bigr),
    \label{eq:cosine}
\end{equation}
so that $\eta_t=\eta$ during the first epoch and $\eta_t\approx6\times10^{-5}\,\eta$ during the last. This is cosine annealing without restarts, warm-up, or a positive floor \citep{loshchilov2017sgdr}, applied once per epoch through the scheduler interface of PyTorch \citep{paszke2019pytorch}.

\paragraph{Minibatch curvature scales.}
As stated in \Cref{sec:dnn}, the $M_L$ scale of a minibatch is $L_0+L_1\norm{\nabla f_i(x)}$ and the $M_H$ scale is $H_0+H_1[f_i(x)-\ell_i^\star]_+$ with $\ell_i^\star=0$, where $f_i$ and $\nabla f_i$ are the minibatch loss and gradient. Note that multiplying $(\rho,\gamma,L_0,L_1)$, or $(\rho,\gamma,H_0,H_1)$, by one positive constant leaves \ref{eq:cssam} unchanged, since the perturbation step $\rho/M(x_t)$ and the descent step $\gamma/M(y_t)$ are ratios. Only three of the four coefficients of each scale are therefore free, and we fix $L_1=1$ for the $M_L$ scale and $H_0=1$ for the $M_H$ scale, so that the tuned coefficients are $(\rho,\gamma,L_0)$ and $(\rho,\gamma,H_1)$ respectively. The coefficients are constant over training; \ref{eq:cssam} uses no learning-rate schedule, since the perturbation and the descent step adapt through the scale.

\paragraph{Tuning protocol.}
The protocol has three stages, each run on all four model--dataset settings.
\begin{enumerate}[label=(\roman*)]
    \item \emph{Baselines.} \ref{eq:sam} and \ref{eq:usam}, each with a constant and with a cosine annealing schedule, are tuned on seed 3107 over a $4\times4$ grid of learning rate and radius (multiplier): $\eta\in\{0.03,0.1,0.3,1\}$ for all four baselines, $r\in\{0.01,0.05,0.1,0.2\}$ for \ref{eq:sam}, and $\rho\in\{0.1,0.3,1,3\}$ for \ref{eq:usam}. For the cosine annealing schedule \eqref{eq:cosine}, $\eta$ is the initial learning rate. 
    
    \item \emph{\ref{eq:cssam}.} The $M_L$ scale is tuned on seed 3107 over $\rho\in\{0.2,0.4,0.8,1.6\}$, $\gamma\in\{0.2,0.4,0.8\}$, and $L_0\in\{0.02,0.5,2\}$ with $L_1=1$, and the $M_H$ scale over $\rho\in\{0.3,1,3\}$, $\gamma\in\{0.3,1,3\}$, and $H_1\in\{1,3,10,30\}$ with $H_0=1$. This stage comprises 36 configurations per scale and setting, that is, 288 runs. 
    
    \item \emph{Seeds.} For every method and setting, the configuration with the highest best test accuracy over the 200 epochs on seed 3107 is rerun with seeds 3407 and 4096. \Cref{tab:cifar} reports the mean and the sample standard deviation over the three seeds of the best test accuracy over epochs.
\end{enumerate}

\paragraph{Selected configurations.}
\Cref{tab:selected} lists, for every method and setting, the configuration selected in stages (i) and (ii) and rerun in stage (iii); these are the configurations behind the entries of \Cref{tab:cifar}.

\begin{table}[h]
    \centering
    \caption{Configurations selected on seed 3107 by the best test accuracy over 200 epochs, for every method and setting of \Cref{tab:cifar}. For \ref{eq:sam} and \ref{eq:usam}, $\eta$ is the constant learning rate or the initial learning rate of the cosine annealing schedule \eqref{eq:cosine}, $r$ is the radius of \ref{eq:sam}, and $\rho$ is the multiplier of \ref{eq:usam}. For \ref{eq:cssam}, $\rho$ and $\gamma$ are the perturbation and step multipliers, $L_0$ is the constant term of the $M_L$ scale $L_0+L_1\norm{\nabla f_i(x)}$ with $L_1=1$, and $H_1$ is the loss coefficient of the $M_H$ scale $H_0+H_1[f_i(x)-\ell_i^\star]_+$ with $H_0=1$ and $\ell_i^\star=0$.}
    \label{tab:selected}
    \resizebox{\linewidth}{!}{%
    \begin{tabular}{l l cccc}
    \toprule
    Method & Parameters & ResNet-20 / CIFAR-10 & ResNet-32 / CIFAR-10 & ResNet-20 / CIFAR-100 & ResNet-32 / CIFAR-100 \\
    \midrule
    \ref{eq:sam}, constant lr & $(\eta,r)$ & $(0.3,\,0.2)$ & $(0.3,\,0.2)$ & $(0.1,\,0.2)$ & $(0.1,\,0.2)$ \\
    \addlinespace
    \ref{eq:sam}, cosine & $(\eta,r)$ & $(1,\,0.1)$ & $(0.3,\,0.2)$ & $(0.3,\,0.2)$ & $(0.3,\,0.2)$ \\
    \addlinespace
    \ref{eq:usam}, constant lr & $(\eta,\rho)$ & $(0.3,\,0.3)$ & $(0.3,\,1)$ & $(0.3,\,1)$ & $(0.3,\,0.3)$ \\
    \addlinespace
    \ref{eq:usam}, cosine & $(\eta,\rho)$ & $(0.3,\,0.3)$ & $(0.3,\,0.3)$ & $(0.3,\,0.3)$ & $(0.3,\,0.3)$ \\
    \addlinespace
    \ref{eq:cssam} with $M_L$ & $(\rho,\gamma,L_0)$ & $(0.8,\,0.4,\,0.02)$ & $(0.8,\,0.4,\,0.02)$ & $(0.4,\,0.2,\,0.02)$ & $(0.4,\,0.4,\,0.02)$ \\
    \addlinespace
    \ref{eq:cssam} with $M_H$ & $(\rho,\gamma,H_1)$ & $(3,\,3,\,10)$ & $(3,\,1,\,1)$ & $(3,\,3,\,3)$ & $(3,\,3,\,3)$ \\
    \bottomrule
    \end{tabular}
    }
\end{table}

\end{document}